\documentclass[12pt,letter,reqno]{amsart}

\usepackage[margin=1in]{geometry}

\usepackage{amssymb,amsfonts,amsthm,mathrsfs}
\usepackage{marginnote}
\usepackage{comment}
\usepackage{enumitem}
\usepackage{graphicx}
\usepackage{bm}
\usepackage{xfrac}

\usepackage[dvipsnames]{color}

\usepackage[colorlinks=true, pdfstartview=FitV, linkcolor=black, citecolor=blue, urlcolor=blue]{hyperref}

\def\Xint#1{\mathchoice
	{\XXint\displaystyle\textstyle{#1}}%
	{\XXint\textstyle\scriptstyle{#1}}%
	{\XXint\scriptstyle\scriptscriptstyle{#1}}%
	{\XXint\scriptscriptstyle%
		\scriptscriptstyle{#1}}%
	\!\int}
\def\XXint#1#2#3{{\setbox0=\hbox{$#1{#2#3}{%
				\int}$ }
		\vcenter{\hbox{$#2#3$ }}\kern-.6\wd0}}
\def\barint{\, \Xint -} 
\def\bariint{\barint_{} \kern-.4em \barint}
\def\bariiint{\bariint_{} \kern-.4em \barint}
\renewcommand{\iint}{\int_{}\kern-.34em \int} 
\renewcommand{\iiint}{\iint_{}\kern-.34em \int} 

\DeclareMathAlphabet{\mathcal}{OMS}{cmsy}{m}{n}

\theoremstyle{plain}

\newtheorem{theorem}{Theorem}[section]

\newtheorem{definition}[theorem]{Definition}
\newtheorem{lemma}[theorem]{Lemma}

\newtheorem{proposition}[theorem]{Proposition}

\theoremstyle{definition}
\newtheorem{remark}[theorem]{Remark}

\renewcommand{\:}{\colon}

\renewcommand{\bar}[1]{\overline{#1}}

\newcommand{\ep}{\varepsilon}
\newcommand{\Lip}{{ Lip}}

\allowdisplaybreaks

\title{Well-Posedness of Contact Discontinuity Solutions and Vanishing Pressure Limit for the Aw-Rascle Traffic Flow Model} 

\author{Zijie Deng, Wenjian Peng, Tian-Yi Wang, Haoran Zhang}

\numberwithin{equation}{section}
\setlist[enumerate]{leftmargin=*}

\begin{document}

	\begin{abstract}
		This paper investigates  the well-posedness of contact discontinuity solutions and the vanishing pressure limit for the Aw-Rascle traffic flow model with general pressure functions. The well-posedness problem is formulated as a free boundary problem, where initial discontinuities propagate along linearly degenerate characteristics.
		To address vacuum degeneracy, a condition at density jump points is introduced, ensuring a uniform lower bound for density. The Lagrangian coordinate transformation is applied to fix the contact discontinuity.The well-posedness of contact discontinuity solutions is established, showing that compressive initial data leads to finite-time blow-up of the velocity gradient, while rarefactive initial data ensures global existence. For the vanishing pressure limit, uniform estimates of velocity gradients and density are derived via level set argument. The contact discontinuity solutions of the Aw-Rascle system are shown to converge to those of the pressureless Euler equations, with matched convergence rates for characteristic triangles and discontinuity lines.
		Furthermore, under the conditions of pressure, enhanced regularity in non-discontinuous regions yields convergence of blow-up times.
	\end{abstract}
	
	
	\maketitle
	
	\section{Introduction}
	In this paper, we consider the Aw-Rascle traffic flow model proposed by Aw and Rascle\cite{A. Aw and M. Rascle}:
	\begin{equation}\label{AR}
		\begin{cases}
			\rho _{t} +\left ( \rho u \right ) _{x} =0,\\
			\left ( \rho \left ( u + P \right )  \right ) _{t} +\left ( \rho u\left ( u + P \right ) \right )_{x}=0,
		\end{cases}
	\end{equation}
	where $t>0$ and $x$ represent time and space, $\rho$, $u$, $P$ are density, velocity, and pressure respectively. The pressure $P$ is the function of density $\rho$ and the small parameter $\varepsilon>0$, satisfying $\lim_{\varepsilon \to 0} P\left ( \rho ,\varepsilon  \right ) =0$. Here we consider the general pressure
	\begin{equation}\label{general pressure}
		P\left(\rho,\varepsilon\right)=\varepsilon^{2}p\left(\rho\right).
	\end{equation}
	In the paper, for $\rho>0$, $p\in C^{2}\left(\mathbb{R^{+}}\right)$ satisfies the following conditions:
	\begin{equation}\label{condition of p}
		{p}' \left(\rho\right)>0, \quad
		2{p}' \left(\rho\right)+\rho{p}'' \left(\rho\right)>0,
		\quad \lim_{\rho \to +\infty}p\left(\rho\right) = +\infty,
		\quad \lim_{\rho \to 0}p\left(\rho\right) = k.	
	\end{equation}
	And for different $k$ we have the following two cases:
	
	Case 1: If $k$ is a finite constant, By $\eqref{AR}_{2}-\varepsilon^{2}k\eqref{AR}_{1}$, \eqref{AR} are equivalent to
	\begin{equation*}\label{}
		\begin{cases}
			\rho _{t} +\left ( \rho u \right ) _{x} =0,\\
			\left ( \rho \left ( u + \varepsilon^{2}\left( p\left(\rho\right)-k \right) \right )  \right ) _{t} +\left ( \rho u\left ( u + \varepsilon^{2}\left( p\left(\rho\right)-k \right) \right ) \right )_{x}=0.
		\end{cases}
	\end{equation*}
	Therefore, we can take $k=0$, otherwise we can let $\tilde{p}\left(\rho\right)=p\left(\rho\right)-k$, then the form of \eqref{AR} is unchanged. In this case, we have $\lim_{\rho \to 0}p\left(\rho\right) = 0$, which includes $\gamma$-law: $p\left(\rho\right)=\rho^{\gamma}$ with $\gamma\ge 1$.

	Case 2: If $k=-\infty$, then $\lim_{\rho \to 0}p\left(\rho\right) = -\infty$,
	which includes $p\left(\rho\right)=\ln_{}{\rho} $.
	To define the inital date, we need to introduce  $\Lip_s(\mathbb{R})$ to denote the set of piecewise Lipschitz functions:
	\begin{definition}
		For each $f$ is belongs to $\Lip_s(\mathbb{R})$, there exists a finite set $\{x_i\}_{i=1}^{n}$ of the first kind of discontinuity points, in which $x_i<x_{i+1}$ for any $i=1, \cdots, n-1$, such that $f$ is Lipschitz function on $(x_i,x_{i+1})$. Respectively, $f$ is belongs to $C^1_s$, if $f$ is $C^1$ function on each $(x_i,x_{i+1})$.
	\end{definition}
	Here, $\{x_i\}_{i=1}^{n}$ is called partition points set. By the Lipschitz continuous at each subinterval, at each partition point $x_i$, there exists the left limit   $\lim_{x\to x_{i}^{-}} f(x) = f\left(x_{i}^{-}\right)$ and the right limit $\lim_{x\to x_{i}^{+}} f(x) = f\left(x_{i}^{+}\right)$, but may not be equal. And, $\left[ f\right](x_i) = f\left(x_{i}^{+}\right) - f\left(x_{i}^{-}\right)$ is the jump of $f$ at $x_i$.

	The initial data of \eqref{AR} can be given as
	\begin{equation}\label{intial value of AR}
		\left(\rho,u\right)|_{t=0} = \left(\rho_{0},u_{0}\right),
	\end{equation}
	where both $\rho_{0}$ and $u_{0}$ are bounded, $\rho_{0}\in Lip_{S} \left(\mathbb{R}\right)$ , $u_{0}\in Lip \left(\mathbb{R}\right)$ and $\rho _{0} \ge \underline{\rho_{0}}>0$.
	At each partition point $x_i$ of density $\rho
	_0$, following conditions are introduce for the low bounded of density: 
	\begin{itemize}
		\item  $\varepsilon$-condition: If $\rho_{0}(x_{i}^{-})<\rho_{0}(x_{i}^{+})$ at $x_{i}$, for $x>x_{i}$,
		\begin{equation}\label{assumption1}
			u_{0}\left(x_{i}\right)+\varepsilon^{2}p\left(\rho_{0}\right) \left(x_{i}^{-}\right) > u_{0}\left(x\right).
		\end{equation}
		
		\item $0$-condition: If $\rho_{0}(x_{i}^{-})<\rho_{0}(x_{i}^{+})$ at $x_{i}$, for $x>x_{i}$,
		\begin{equation}\label{assumption2}
			u_{0}\left(x_{i}\right) > u_{0}\left(x\right).
		\end{equation}
	\end{itemize}
	The necessary and sufficient of above conditions will be discussed in Section 3 and Section 4. And, $0$-condition is the joint of $\varepsilon$-condition for $\varepsilon> 0$.
	\begin{remark}
		The discontinuous points of $\rho$ are of the first kind. If the function $f\in \Lip_s(\mathbb{R})$, then $f$ can be decomposed into
		\begin{equation}
			f=f_{J}+f_{C},
		\end{equation}
		where $f_{J}$ represents the jump part of function $f$; $f_{C}$ represents the absolutely continuous part of the function $f$ and is a Lipschitz continuous function.
	\end{remark}
	
	\begin{remark}
		
		If the curve $x\left(t;x_{0},0\right)$ is from the initial discontinuity point $\left(x_{0},0\right)$, the discontinuity will propagate along this curve. On the other hand, according to the Rankine-Hugoniot condition, $\left[ u\right]=0$ leads to 
		\begin{equation}
			\frac{dx}{dt} = \frac{\left[\rho u\right]}{\left[\rho\right]}
			= u.
		\end{equation}
		Then, the discontinuous curve $x\left(t;x_{0},0\right)$ satisfies
		\begin{equation}
			\begin{cases}
				\frac{dx\left(t;x_{0},0\right)}{dt}=u\left(x\left(t;x_{0},0\right),t\right) \\
				x\left(0;x_{0},0\right) = x_{0}.
			\end{cases}
		\end{equation}
	\end{remark}
	
	
	For fixed $\varepsilon>0$, the two eigenvalues of Aw-Rascle model consist: the genuinely nonlinear one, and the linearly degenerate one. For the well-posedness of Aw-Rascle model, we start with local existence and gradient blow-up.

	For the local existence, if the solution is $C^{1}$ function without vacuum, \eqref{AR} is equivalent to strict hyperbolic system. In this case, the local existence of various types of function spaces, including $C^{1}$ function class, $H^{s}$ function class and $BV$ function class, has been studied. Under the condition of $\rho >0$ in the whole space, Li-Yu\cite{Li Ta-Tsien Yu Wenci}  could provide the local well-posedness of one-dimensional $C^{1}$ solutions on each compact characteristic triangle.

	For gradient blow-up of conservation laws, in 1964, Lax\cite{Lax P D} proved this conclusion for one-dimensional $2\times 2$ genuinely nonlinear hyperbolic system. The results show that for a strictly hyperbolic system, if the initial value is a small smooth perturbation near a constant state, then the initial compression in any genuinely nonlinear characteristic field will produce a gradient blow-up in  finite time. John\cite{John F}, Li-Zhou-Kong\cite{Li Ta-Tsien Zhou Yi Kong Dexing, Li Ta-Tsien Zhou Yi Kong Dexing 2}, Liu\cite{Liu Taiping} have proven the generation of shock waves for $n\times n$ conservation law equations under different conditions.

	The Riemann problem of \eqref{AR} with $p\left(\rho\right) = \rho^{\gamma}$ was solved in \cite{A. Aw and M. Rascle}. Pan-Han\cite{L. Pan X. Han} introduced the Chaplygin pressure function into the Aw–Rascle traffic model and solved the respective Riemann problem. It is noticeable that  under the generalized Rankine–Hugoniot condition and entropy condition, they establish the existence and uniqueness of $\delta$-wave. Cheng-Yang\cite{H. Cheng and H. Yang} solved the Riemann problem for the Aw–Rascle model with the modified Chaplygin gas pressure. Godvik and Hanche-Olsen\cite{M. Godvik and H. Hanche-Olsen} proved the existence of the weak entropy solution for the Cauchy problem with vacuum. By the compensated compactness method, Lu\cite{Y. Lu} proved the global existence of bounded entropy weak solutions for the Cauchy problem of general nonsymmetric systems of Keyfitz–Kranzer type. When the parameter $n=1$, the system is the Aw–Rascle model.

	When the pressure $P\left(\rho\right)\equiv 0$ in the Aw-Rascle model, the equations are simplified to pressureless Euler equations:
	\begin{equation}\label{pressureless Euler}
		\begin{cases}
			\rho _{t} +\left ( \rho u \right ) _{x} =0,\\
			\left ( \rho u \right ) _{t} +\left ( \rho u^{2} \right )_{x}=0,
		\end{cases}
	\end{equation}
	The system (1.3) was used to describe both the process of the motion of free particles sticking under collision\cite{Brenier Grenier} and the formation of large scale in the universe\cite{E Weinan Rykov Sinai,Shandarin}. The research on the pressureless Euler equations mainly focuses on the well-posedness of weak solutions. Brenier-Grenier\cite{Brenier Grenier} and E-Sinai\cite{E Weinan Rykov Sinai} independently obtained the existence of global weak solutions, and E-Sinai obtained the explicit expression of weak solutions by using the generalized variational principle. Wang-Huang-Ding\cite{Wang Zhen Huang Feimin Ding Xiaxi} prove the global existence of generalized solution to the $L^\infty$ initial data. Boudin\cite{Boudin} proved that the weak solution is the limit of the solution of the viscous pressureless Euler equations. Furthermore, Huang\cite{Huang Feimin} proved the existence of entropy solutions for general pressureless Euler equations.
	Wang-Ding \cite{Wang Zhen Ding Xiaxi} proved the uniqueness of the weak solution of the Cauchy problem satisfying the Oleinik entropy condition when the initial value $\rho_{0}$ is a bounded measurable function. Bouchut-James\cite{Bouchut F James F.} also obtained similar results. Huang-Wang\cite{Huang Feimin Wang Zhen} proved the uniqueness of the weak solution when the initial value is the Radon measure.
	
	To consider the relationship between the Aw-Rascle model and the pressureless Euler equations, a natural idea is the vanishing pressure  limit, which considers limits of  $\varepsilon\rightarrow0$ as the pressure in the form of \eqref{general pressure}. The first study of vanishing pressure  limit is on the Riemann solution of the isentropic Euler equations by Chen-Liu\cite{Chen Gui-Qiang Liu Hailiang},  the limiting solution in which includes $\delta$-wave by concentration and vacuum by cavitation. Later, they  extended the above result to the full Euler case\cite{Chen Gui-Qiang Liu Hailiang 2}. Recently, Peng-Wang\cite{Peng Wenjian Wang Tian-Yi} studied the case of  $C^1$ solutions by a new level set argument. They showed that: For compressive initial data, the continuous solutions converge to a mass-concentrated solution of the pressureless Euler system; For rarefaction initial data, the solutions instead converge globally to a continuous solution. In \cite{Peng Wenjian Wang Tian-Yi Xiang Wei}, the authors studied the hypersonic limitation for $C^1$ solution, and showed the convergence of blow-up time. For the Aw-Rascle model, Shen and Sun \cite{C. Shen M. Sun} proved that as $\varepsilon\rightarrow0$, the Riemann solutions of the perturbed Aw-Rascle system converge to the ones of the pressureless Euler equations \eqref{pressureless Euler}. Pan-Han\cite{L. Pan X. Han}  proved that for the Riemann problem, as the Chaplygin pressure $P\left(\rho\right)=-\frac{\varepsilon}{\rho}$ vanishes, the Riemann solutions of the Aw-Rascle traffic model  converge to the respective solutions of the pressureless gas dynamics model \eqref{pressureless Euler}.


	This paper addresses two fundamental problems for the Aw-Rascle traffic flow model: (i) the well-posedness of contact discontinuity solutions  with initial velocity  $u_{0}$ and piecewise Lipschitz initial density $\rho_{0}$ satisfying the $\varepsilon$-condition, and (ii) the vanishing pressure limit as $\varepsilon\to 0$. For problem (i), the analysis begins by establishing a strictly positive lower bound for the density through a partition of the domain, and Lagrangian coordinate transformations. This lower bound leads to a dichotomy: compressive initial data induce finite time velocity gradient blow-up, whereas rarefactive initial data guarantee global existence of solutions. For problem (ii), uniform density estimates independent of $\varepsilon$ are derived via level set argument. It is proven that Aw-Rascle solutions converge to pressureless Euler solutions as  $\varepsilon\to 0$: compressive data lead to mass-concentrated solutions, while rarefactive data yield globally regular solutions, with matching 
	$O(\varepsilon^2)$ convergence rates for velocity fields and characteristic triangles. Furthermore, under the conditions of  pressure, convergence of the blow-up time is established through enhanced regularity analysis in non-discontinuous regions.

	We have the following two theorems.

	\begin{theorem}\label{theorem 1}
		For fixed $\varepsilon>0$, and \eqref{AR} - \eqref{intial value of AR}, if $\lim_{\rho \to 0}p\left(\rho\right)$ satisfies one of the following two cases: 
		
		\noindent Case 1: $\lim_{\rho \to 0}p\left(\rho\right)=0$ and the initial data satisfies the $\varepsilon$-condition; 
		
		\noindent Case 2: $\lim_{\rho \to 0}p\left(\rho\right)=-\infty$;
		
		\noindent then there exists a time $T_{b}^{\varepsilon}$ such that on $\mathbb{R}\times\left[0,T_{b}^{\varepsilon}\right)$, there exists contacted discontinuity solution $\left(\rho^{\varepsilon},u^{\varepsilon}\right)$ satisfying 
		
		(1) If $\inf_{x\in\mathbb{R}} u_{0}'\left(x\right)\ge 0$, $T^{\varepsilon}_{b}=+\infty$, the solution exists globally, 
		\begin{equation*}
			\rho^{\varepsilon} \in Lip_s\left(\mathbb{R}\times\left[0,+\infty\right)\right),\quad
			u^{\varepsilon}\in Lip \left(\mathbb{R}\times\left[0,+\infty\right)\right).
		\end{equation*}

		(2) If $\inf_{x\in\mathbb{R}} u_{0}'\left(x\right)<0$, $T_{b}^{\varepsilon}$ is finite and there exists at least one $X_{b}^{\varepsilon}$ such that as $\left(x,t\right)\to \left(X_{b}^{\varepsilon},T_{b}^{\varepsilon}\right)$, $u_{x}^{\varepsilon}\left(x,t\right)\to -\infty$ while $\rho^{\varepsilon}\left(x,t\right)$ is upper and lower bounded. And the solution stands 
		\begin{equation*}
			\rho^{\varepsilon} \in Lip_s \left(\mathbb{R}\times\left[0,T_{b}^{\varepsilon}\right)\right),\quad
			u^{\varepsilon}\in Lip \left(\mathbb{R}\times\left[0,T_{b}^{\varepsilon}\right)\right).
		\end{equation*} 
		And, $\Gamma^{\varepsilon} = \left \{ \left ( x,t \right ) \in \mathbb{R}\times\left[0,T_{b}^{\varepsilon}\right) : x=x_{2}^{\varepsilon}\left(t\right) \right \} $ is the discontinuous curve and $\left(x_{2}^{\varepsilon}\right)'\left(t\right)$ is a Lipschitz function with respect to $t$, satisfying
		\begin{equation*}
			\frac{dx_{2}^{\varepsilon}\left(t\right)}{dt} = u^{\varepsilon}\left(x_{2}^{\varepsilon}\left(t\right),t\right) .
		\end{equation*}
	\end{theorem}
	
	\begin{theorem}\label{Vanishing Pressure Limit}
		For $\varepsilon>0$, $\left(\rho^{\varepsilon},u^{\varepsilon}\right)$ are the unique solution of \eqref{AR} - \eqref{intial value of AR} with $0-$condition on $\mathbb{R} \times\left[0,T_{b}^{\varepsilon}\right)$.
		And,  $\left(\bar{\rho},\bar{u}\right)$ are the unique solution of \eqref{pressureless Euler} with initial data \eqref{intial value of AR} on $\mathbb{R} \times\left[0,T_{b}\right)$.
		
		(1) For any $0<T_{\ast } < T_{b}$, there exists a $\varepsilon_{\ast}>0$, such that, for $0<\varepsilon<\varepsilon_{\ast}$, $T_{\ast}<T_{b}^{\varepsilon}$. As $\varepsilon\to 0$, 
		\begin{equation*}
			\rho^{\varepsilon}\to \bar{\rho} \,\,{\rm in} \,\, \mathcal{M} \left(\mathbb{R}\times\left[0,T_{\ast}\right] \right),\quad
			u^{\varepsilon}\to \bar{u} \,\,{\rm in} \,\, Lip \left(\mathbb{R}\times\left[0,T_{\ast}\right]\right),
		\end{equation*}
		And, $	T_{b} \le \varliminf_{\varepsilon\to0} T_{b}^{\varepsilon}$. Furthermore, there are the following convergence rates: For $i=1,2$, 
		$$\left | u^{\varepsilon } - \bar{u} \right | \sim O\left ( \varepsilon ^{2} \right ), 
		\quad \left | \lambda_{i}^{\varepsilon } - \bar{u} \right | \sim O\left ( \varepsilon ^{2} \right ), \quad \left | x_{2}^{\varepsilon } - \bar{x} \right | \sim O\left ( \varepsilon ^{2} \right ).$$
		where $\lambda_{i}$ are the eigenvalues of \eqref{AR}; $x_{2}^{\varepsilon}$ and $\bar{x}$ are the discontinuity lines of \eqref{AR} and \eqref{pressureless Euler} respectively.
		
		(2) If $\rho_{0}\in C^1_{s} \left(\mathbb{R}\right)$, $u_{0}\in C^{1} \left(\mathbb{R}\right)$, $I\left(\rho\right):=\frac{2\rho{p}'\left(\rho\right)+\left(\rho\right)^{2}{p}'' \left ( \rho \right ) }{\left(\rho{p}'\left(\rho\right)\right)^{2}}$ satisfies following conditions:
		\begin{enumerate}[label=(\alph*)]
			\item $I\left(\rho\right)$ is an increasing function with respect to $\rho$;
			\item There exists a small $\delta$ such that $\int_{0}^{\delta}\frac{I(s)}{s^{2}}ds=+\infty$.
		\end{enumerate}
		Then, the blow-up time  convergences: $\lim_{\varepsilon \to 0} T _{b}^{\varepsilon} = T_{b}$.
	\end{theorem}
	
	Comparing with classical solutions of the compressible Euler equations, contact discontinuity solutions of the Aw-Rascle model and their vanishing pressure limits pose unique analytical challenges. First, the characteristic structure of the Aw-Rascle system, comprising both a linearly degenerate field and a genuinely nonlinear field, induces distinct regularity properties in Riemann invariants, with initial discontinuities propagating along linearly degenerate characteristic curves. Second, the coupling between evolving discontinuity curves and the solution itself characterizes the problem as a free boundary problem. Third, density-dependent degeneracy in the Riccati-type equations governing the system necessitates new density lower-bound estimation, distinct from those for the compressible Euler equations. Fourth, discontinuities inherently reduce solution regularity, mandating uniform estimates in tailored function spaces to rigorously establish convergence. Finally, precise analysis of blow-up times requires delicate estimates in non-discontinuity regions, where enhanced regularity can be exploited.
	
	To overcome these challenges, three key ideas are introduced:
	(i) Lagrangian coordinate transformations: that map evolving discontinuity curves to the fixed boundaries; (ii) Density Lower bound analysis: the derivatives of Riemann invariants in smooth regions and the jump condition analysis at discontinuities;
	(iii) Time-directional derivative techniques: avoiding spatial discontinuities in Lagrangian. Level set argument are further developed to track invariant derivatives and establish uniform estimates for the vanishing pressure limit.
	
	The paper is organized as follows. Section 2 proves the existence of classical solutions to the Aw-Rascle system through Lagrangian coordinate transformations, on avoiding vacuum formation. Section 3 establishes the well-posedness of contact discontinuity solutions using uniform $L^\infty$  estimates for time-directional derivatives in Lagrangian coordinates. Section 4 studies the vanishing pressure limit, while Section 5 quantifies blow-up time convergence via modulus of continuity estimates in non-discontinuity domains. Finally, Section 6 determines sharp convergence rates ($O(\varepsilon^2)$ ) for characteristic curves and Riemann invariants.

	\section{The well-posedness of Aw-Rascle model with $C^{1}$ initial data}
	\subsection{Lagrangian coordinates}
	In this section, we consider the Lagrangian transformations for fixed $\varepsilon>0$
	that does not involve estimations, so we drop $\varepsilon$ of  $\left(\rho^{\varepsilon},u^{\varepsilon}\right)$.
	
	The eigenvalues of \eqref{AR} are
	\begin{equation} \label{eigenvalues}
		\begin{cases}
			\lambda_{1} =  u  -\varepsilon^{2}\rho {p}' \left(\rho\right) ,\\
			\lambda_{2} = u,
		\end{cases}
	\end{equation}
	and  Riemann invariants are
	\begin{equation}\label{Riemann invariants}
		\begin{cases}
			w = u ,\\
			z = u + \varepsilon^{2} p\left(\rho \right),
		\end{cases}
	\end{equation}
	while $p\left(\rho\right)$ can be represented by Riemann invariants as $p\left(\rho\right) = \frac{z -u }{\varepsilon^{2} }$.
	Then, for $\rho>0$, \eqref{AR} is equivalent to
	\begin{equation}\label{equation of Riemann invariants 2} 
		\begin{cases}
			u_{t} + \left(u -\varepsilon^{2}\rho{p}' \left(\rho\right)\right)u_{x} = 0, \\
			\left( u + \varepsilon^{2} p\left(\rho \right)\right)_{t} + 
			u \left( u + \varepsilon^{2} p\left(\rho \right)\right)_{x}=0.
		\end{cases}
	\end{equation}
	The derivatives along the charismatics are: for $i=1,2$,
	\begin{equation} 
		D_{i} = \partial_{t} + \lambda_{i} \partial_{x},
	\end{equation} 
	and the respective characteristic lines passing through $\left(\tilde{x},\tilde{t}\right)$ are defined as:
	\begin{equation}
		\begin{cases}\label{x-i}
			\frac{dx_{i}\left(t;\tilde{x},\tilde{t}\right)}{dt} = \lambda_{i} \left(x_{i}\left(t;\tilde{x},\tilde{t}\right),t\right) , \\
			x_{i}\left(\tilde{t};\tilde{x},\tilde{t}\right) = \tilde{x}.
		\end{cases}
	\end{equation}
	Combining \eqref{eigenvalues} and \eqref{equation of Riemann invariants 2}, we have
	\begin{equation}\label{}
		\begin{cases}
			D_{1} u = 0, \\
			D_{2} z = 0.
		\end{cases}
	\end{equation}
	Since the eigenvalue $\lambda_{1}$ is genuinely nonlinear and $\lambda_{2}$ is linearly degenerate, there will be contact discontinuity in density along the second family of eigenvalues. So we consider the following Lagrangian transformation: let $\tau=t$ and
	\begin{equation} \label{Lagrangian transformation}
		\begin{cases}
			\frac{d}{d\tau}x\left(\tau ; y\right) = u\left(x\left(\tau;y\right),\tau\right), \\
			x\left(0;y\right)=y.
		\end{cases}
	\end{equation}
	Let $u\left(x\left(\tau;y\right),\tau\right)=v\left(y,\tau\right)$, $J = \frac{\partial x}{\partial y}$ is the Jacobian determinant of the coordinate transformation satisfying:
	\begin{equation} \label{v and vy}
		\begin{cases}
			\frac{\partial }{\partial\tau}J = v_{y} \left(y,\tau\right), \\
			J \left(y,0\right)=1.
		\end{cases}
	\end{equation}
	In Lagrangian coordinates, we denote 
	\begin{equation}\label{variables in Lagrangian coordinates}
		g\left(y,\tau\right):=
		\rho\left(x\left(\tau;y\right),\tau\right),\quad
		Z\left(y,\tau\right): =z\left(x\left(\tau;y\right),\tau\right),
	\end{equation}
	and $Z\left(y,\tau\right) = v\left(y,\tau\right) + \varepsilon^{2} p\left(g\left(y,\tau\right)\right)$.
	
	Next, for the relationship between $J$ and $g$, in Lagrangian coordinates, $\eqref{AR}_{1}$ is equivalent to
	\begin{equation}
		g_{\tau} + g J^{-1}J_{\tau}=0,
	\end{equation}
	which equals to $ {\left (\ln (g J) \right ) } _{\tau } =0$.
	For $g>0$, integrating on $\tau$ leads to
	\begin{equation}\label{J}
		J\left(y,\tau\right) = \frac{g_{0}\left(y\right)}{g\left(y,\tau\right)},
	\end{equation}
	where $g_{0}(y)=g(y,0)$ is the initial density.
	Combining with \eqref{v and vy}, in Lagrangian coordinates, \eqref{equation of Riemann invariants 2} is equivalent to
	\begin{equation}\label{equation of Lagrangian transformation 2}
		\begin{cases}
			J _{\tau} = v _{y},\\
			v _{\tau} + \mu v _{y}=0, \\
			Z _{\tau}=0,
		\end{cases}
	\end{equation}
	where $\mu = -\varepsilon^{2}g{p}' \left(g\right) J^{-1} = -\frac{\varepsilon^{2}}{g_{0}}g^{2}{p}' \left(g\right)$. And, the initial data are
	\begin{equation}\label{new initial value}
		\left ( J, v ,Z \right ) |_{\tau=0} =\left ( 1, v_{0}\left(y\right), Z_{0} \left(y\right) \right ).
	\end{equation}

	In the  Lagrangian coordinates, one could introduce following direction derivative:
	\begin{equation} 
		D = \partial_{\tau} + \mu \partial_{y},
	\end{equation} 
	and the characteristic line passing through $\left(\tilde{y},\tilde{\tau}\right)$
	\begin{equation} \label{characteristic line}
		\begin{cases}
			\frac{dy_{1}\left(\tau;\tilde{y},\tilde{\tau}\right)}{d\tau} = \mu \left(y_{1}\left(\tau;\tilde{y},\tilde{\tau}\right),\tau\right) , \\
			y_{1}\left(\tilde{\tau};\tilde{y},\tilde{\tau}\right) = \tilde{y}.
		\end{cases}
	\end{equation}
	Along the characteristic line, $\eqref{equation of Lagrangian transformation 2}_{2}$ is equivalent to 
	\begin{equation}\label{transport of v}
		Dv = 0, \\
	\end{equation}
	which leads to
	\begin{equation}\label{v=v0} 
		v\left(y,\tau\right) = v_{0}\left(y_{1}\left(0;y,\tau\right)\right).
	\end{equation}
	And, by $\eqref{equation of Lagrangian transformation 2}_{3}$, one could have
	\begin{equation}\label{Z=Z0}
		Z(y,\tau) = Z_{0}(y).
	\end{equation}
	By the expression of $p$ we have
	\begin{equation}\label{equation of p}
		p\left(g(y,\tau)\right) =\frac{Z(y,\tau)-v(y,\tau)}{\varepsilon^{2}} = \frac{Z_{0}(y)-v_{0}\left(y_{1}\left(0;y,\tau\right)\right)}{\varepsilon^{2}} .
	\end{equation}
	If $g(y,\tau)>0$
	\begin{equation}
		g = p^{-1}\left(p\left(g_{0}(y)\right) + \frac{v_{0}(y)-v\left(y,\tau\right)}{\varepsilon^{2}}\right) .
	\end{equation}
	\begin{remark}\label{single equation}
		If absence of vacuum ($g(y,\tau)>0$), the density $g$  are uniquely determined by the velocity field $v(y,\tau)$ and Lagrangian coordinate $y$. Consequently, the velocity field $v$ is governed by the transport equation:
		\begin{equation}\label{equation of single v}
			v_{\tau} + \mu\left(v,y\right) v_{y} = 0.
		\end{equation}
		For initial data containing jump, the coefficient $\mu (v,y)$ becomes discontinuous along $y$-direction. This necessitates estimating $v$ through time $\tau$-derivatives via Lagrangian evolution instead of spatial $y$-derivatives, thereby avoiding  the jumps. Since the $\varepsilon$-condition in this paper ensures that the density has a lower bound at jump. The structure of \eqref{equation of single v} is the basis of the regularity of the solution in this paper. Even in the presence of contact discontinuities, its structure enables deriving the appropriate gradient estimates.
		
	\end{remark}
	
	And in Lagrangian coordinates, $\varepsilon$-condition and $0$-condition equivalence to:
	\begin{itemize}
		\item $\varepsilon$-condition: If $g_{0}(y_{i}^{-})<g_{0}(y_{i}^{+})$ at jump point $y_{i}$,  for $y>y_{i}$,
		\begin{equation}\label{assumption1L}
			v_{0}(y_{i})+\varepsilon^{2}p\left(g_{0}\right)\left(y_{i}^{-}\right) > v_{0}\left(y\right).
		\end{equation}
		
		\item  $0$-condition: If $g_{0}(y_{i}^{-})<g_{0}(y_{i}^{+})$ at jump point $y_{i}$,  for $y>y_{i}$
		\begin{equation}\label{assumption2L}
			v_{0}\left(y_{i}\right) > y_{0}\left(y\right).
		\end{equation}
	\end{itemize}
	In the paper, we do not distinguish the $\varepsilon$-condition and $0$-condition in the Eulerian coordinates or Lagrangian coordinates.

	\subsection{The well-posedness and blow-up $C^{1}$ solutions}
	\ 
	\newline
	\indent
	In this section, we consider $v_{0}$, $Z_{0} \in C^{1}$, by \cite{Li Ta-Tsien Yu Wenci}, we could have the local existence of $C^{1}$ solution with $g(y,\tau)>0$ for \eqref{equation of Lagrangian transformation 2}-\eqref{new initial value}. 
	Next, we further consider the sharp life-span of $C^{1}$ solution of \eqref{AR} - \eqref{condition of p}. First, we have the following lemma.
	\begin{lemma}
		For the $C^{1}$ solution $\left(J^{\varepsilon},v^{\varepsilon},Z^{\varepsilon}\right)$ of \eqref{equation of Lagrangian transformation 2}-\eqref{new initial value},  
		$g^{\varepsilon}$ is upper bounded with respect to $\varepsilon$; $v^{\varepsilon}$ is uniformly bounded with respect to $\varepsilon$.
		
	\end{lemma}
	\begin{proof}
		By \eqref{v=v0},
		\begin{equation}
			\min_{y\in\mathbb{R}} v_{0}(y) \le v^{\varepsilon}(y,\tau)\le \max_{y\in\mathbb{R}}v_{0}(y).
		\end{equation}
		On the other hand, for $p\left(g^{\varepsilon}\right)$
		\begin{eqnarray}\label{upper bound of p}
			p\left(g^{\varepsilon}\left(y,\tau\right)\right) & = & \frac{Z^{\varepsilon}\left(y,\tau\right) -v^{\varepsilon}\left(y,\tau\right)}{\varepsilon^{2}} \nonumber \\
			& = & \frac{\varepsilon^{2}p\left(g_{0}\left(y\right)\right) +v_{0}\left(y\right) -v_{0}\left(y_{1}^{\varepsilon}\left(0;y,\tau\right)\right)}{\varepsilon^{2}} \nonumber \\
			& \le & \max_{y\in\mathbb{R}} p \left(g_{0}\right) + \frac{2\max_{y\in\mathbb{R}}v_{0}}{\varepsilon^{2}} \le C\left(\varepsilon\right).
		\end{eqnarray}
		Combining \eqref{condition of p}, $g^{\varepsilon}$ has an upper bound with respect to $\varepsilon$, denoted by $\bar{g}^{\varepsilon}$.
	\end{proof}

	We have the following lower bound estimate for density $g^{\varepsilon}$.
	\begin{proposition}\label{lower bound of C1 initial value}
		For the $C^{1}$ solutions $\left(J^{\varepsilon},v^{\varepsilon},Z^{\varepsilon}\right)$ of \eqref{equation of Lagrangian transformation 2}-\eqref{new initial value},  $g^{\varepsilon}$ has a uniform lower bound
		\begin{equation}\label{the density lower bound of C1 initial value}
			g^{\varepsilon}\left(y,\tau\right)\ge \frac{A_{1}}{1+A_{2}B\tau}.
		\end{equation}
		where $A_{1}=\min_{y\in\mathbb{R}} g_{0}$, $A_{2}=\max_{y\in\mathbb{R}} g_{0}$, $B = \max_{y\in\mathbb{R}} \frac{\left(\left(Z_{0}^{\varepsilon}\right)'\right)_+}{g_{0}} $ with $(f)_+=\max(f, 0)$.
	\end{proposition}
	\begin{proof}
		To estimate $g^{\varepsilon}$
		\begin{equation}
			p'\left(g^{\varepsilon}\right)Dg^{\varepsilon} = Dp\left(g^{\varepsilon}\right) = \frac{1}{\varepsilon^{2}}DZ^{\varepsilon} = \frac{1}{\varepsilon^{2}}\left(Z^{\varepsilon}_{\tau} + \mu^{\varepsilon}Z^{\varepsilon}_{y}\right)  = \frac{1}{\varepsilon^{2}} \mu^{\varepsilon}Z^{\varepsilon}_{y}= -g_{0}^{-1}\left(g^{\varepsilon}\right)^{2}p'\left(g^{\varepsilon}\right)\left(Z^{\varepsilon}_{0}\right)_{y}.
		\end{equation}
		Then, we have
		\begin{equation}
			Dg^{\varepsilon} = -g_{0}^{-1}\left(g^{\varepsilon}\right)^{2}\left(Z^{\varepsilon}_{0}\right)_{y}.
		\end{equation}
		Dividing both sides of the above equation by $\left(g^{\varepsilon}\right)^{2}$, by the definition of $B$, we have
		\begin{equation}\label{D(1/g)}
			D\left(\frac{1}{g^{\varepsilon}}\right) = \frac{\left(Z_{0}^{\varepsilon}\right)_{y}}{g_{0}} \le B.
		\end{equation}
		Integrating $s$ from 0 to $\tau$ along the characteristic line $y_{1}^{\varepsilon}\left(\tau;y,\tau\right)$
		\begin{equation}
			g^{\varepsilon}\left(y,\tau\right)\ge\frac{g_{0}\left(y^{\varepsilon}_{1}\left(0;y,\tau\right)\right)}{1+g_{0}\left(y^{\varepsilon}_{1}\left(0;y,\tau\right)\right)B\tau} \ge \frac{A_{1}}{1+A_{2}B\tau}.
		\end{equation}
	\end{proof}

	
	So for the $C^{1}$ initial data, we have the following proposition
	\begin{proposition}\label{life-span of C^1 solution}
		For fixed $\varepsilon>0$ and \eqref{AR} - \eqref{condition of p}, $\left(\rho_{0},u_{0}\right)\in \left(C^{1}\left(\mathbb{R}\right)\right)^{2}$, there exists a time $T_{b}^{\varepsilon}$ such that on $\mathbb{R}\times\left[0,T_{b}^{\varepsilon}\right)$:
		
		(1) If $\inf_{x\in\mathbb{R}} u_{0}'\left(x\right)\ge 0$, $T^{\varepsilon}_{b}=+\infty$, the solution exists globally, $(\rho^{\varepsilon},u^{\varepsilon} )\in(C^{1}\left(\mathbb{R}\times\left[0,+\infty\right)\right))^2$.
		
		(2) If $\inf_{x\in\mathbb{R}} u_{0}'\left(x\right)<0$, $T_{b}^{\varepsilon}$ is finite and there exists at least one $X_{b}^{\varepsilon}$ such that as $\left(x,t\right)\to \left(X_{b}^{\varepsilon},T_{b}^{\varepsilon}\right)$, $u_{x}^{\varepsilon}\left(x,t\right)\to -\infty$ while $\rho^{\varepsilon}\left(x,t\right)$ is upper and lower bounded. And the solution stands $(\rho^{\varepsilon},u^{\varepsilon} )\in(C^{1}\left(\mathbb{R}\times\left[0,T_{b}^{\varepsilon}\right)\right))^2$.
		
	\end{proposition}
	\begin{proof}
		Since the lower bound of density, by Remark \ref{single equation}, we just need to consider the gradient of $v^{\varepsilon}$.
		Taking $\partial_{\tau}$ on $\eqref{equation of Lagrangian transformation 2}_{2}$  leads to
		\begin{equation}\label{Dvtau}
			D\left(v^{\varepsilon}_{\tau}\right) = \left(v^{\varepsilon}_{\tau}\right)_{\tau}+\mu^{\varepsilon}\left(v^{\varepsilon}_{\tau}\right)_{y} =  - \mu^{\varepsilon}_{\tau} v^{\varepsilon}_{y}.
		\end{equation}
		To compute $\mu^{\varepsilon}_{\tau}v^{\varepsilon}_{y}$, we have
		\begin{equation}
			g^{\varepsilon}_{\tau}=-g^{\varepsilon}\left(J^{\varepsilon}\right)^{-1}v^{\varepsilon}_{y} = -\frac{\left(g^{\varepsilon}\right)^{2}v^{\varepsilon}_{y}}{g_{0}}
			= \frac{\left(g^{\varepsilon}\right)^{2}v^{\varepsilon}_{\tau}}{g_{0}\mu^{\varepsilon}}
			= -\frac{v^{\varepsilon}_{\tau}}{\varepsilon^{2}p'\left(g^{\varepsilon}\right)},
		\end{equation}
		which indicates
		\begin{eqnarray}\label{I(g)}
			\mu^{\varepsilon}_{\tau}v^{\varepsilon}_{y} & = & -\frac{\varepsilon^{2}}{g_{0}}\left(\left(g^{\varepsilon}\right)^{2}p'\left(g^{\varepsilon}\right)\right)_{\tau}v^{\varepsilon}_{y} \nonumber \\
			& = & -\frac{\varepsilon^{2}}{g_{0}}\left(2g^{\varepsilon}p'\left(g^{\varepsilon}\right)+\left(g^{\varepsilon}\right)^{2}p''\left(g^{\varepsilon}\right)\right)g^{\varepsilon}_{\tau}v^{\varepsilon}_{y} \nonumber \\
			& = & \frac{2g^{\varepsilon}{p}'\left(g^{\varepsilon}\right)+\left(g^{\varepsilon}\right)^{2}{p}'' \left ( g^{\varepsilon} \right ) }{\varepsilon^{2}\left(g^{\varepsilon}{p}'\left(g^{\varepsilon}\right)\right)^{2}}\left(v^{\varepsilon}_{\tau}\right)^{2}.
		\end{eqnarray}
		By $I\left(g\right):=\frac{2g{p}'\left(g\right)+\left(g\right)^{2}{p}'' \left ( g \right ) }{\left(g{p}'\left(g\right)\right)^{2}}$, one could get
		\begin{equation}\label{riccati}
			D\left(v^{\varepsilon}_{\tau}\right)=-\frac{I\left(g^{\varepsilon}\right)}{\varepsilon^{2}}\left(v^{\varepsilon}_{\tau}\right)^{2}.
		\end{equation}
		Since $D\left(v^{\varepsilon}_{\tau}\right)\le 0$, $v^{\varepsilon}_{\tau}$ is upper bounded: $v^{\varepsilon}_{\tau}(y,\tau)\le \max_{y\in \mathbb{R}}v^{\varepsilon}_{\tau}(y,0)$. Next, we focus on the lower bound of $v^{\varepsilon}_{\tau}$.
		Dividing both sides of the above formula by $\left(v^{\varepsilon}_{\tau}\right)^{2}$, we have
		\begin{equation}
			D\left(-\frac{1}{v^{\varepsilon}_{\tau}}\right)=
			D\left(-\frac{1}{\varepsilon^{2}g^{\varepsilon}p'\left(g^{\varepsilon}\right)\left(J^{\varepsilon}\right)^{-1}v^{\varepsilon}_{y}}\right) =
			-\frac{I\left(g^{\varepsilon}\right)}{\varepsilon^{2}}.
		\end{equation}
		Integrating $s$ from 0 to $\tau$ along the characteristic line $y_{1}^{\varepsilon}\left(\tau;\xi,0\right)$
		\begin{equation}\label{blow up term}
			\frac{1}{\left(g^{\varepsilon}p'\left(g^{\varepsilon}\right)\left(J^{\varepsilon}\right)^{-1} v^{\varepsilon}_{y}\right) \left(y_{1}^{\varepsilon}\left(\tau;\xi,0\right),\tau\right)}
			-\frac{1}{\left(g_{0}p'\left(g_{0}\right)v_{0}'\right)\left(\xi\right)}= \frac{1}{\varepsilon^{2}}
			\int_{0}^{\tau}I\left(g^{\varepsilon}\left(y_{1}^{\varepsilon}\left(s;\xi,0\right),s\right)\right)ds.
		\end{equation}
		By \eqref{blow up term}, we have
		\begin{equation}\label{blow-up of general pressure}
			\left(g^{\varepsilon}p'\left(g^{\varepsilon}\right)\left(J^{\varepsilon}\right)^{-1}v_{y}^{\varepsilon}\right)\left(y_{1}^{\varepsilon}\left(\tau;\xi,0\right),\tau\right) \\
			= \frac{\left(g_{0} {p}'\left(g_{0}\right)v_{0}'\right)\left(\xi\right)}
			{1+\left(g_{0}{p}'\left(g_{0}\right)v_{0}'\right)\left(\xi\right)\int_{0}^{\tau}I\left(g^{\varepsilon}\left(y_{1}^{\varepsilon}\left(s;\xi,0\right),s\right)\right)ds} .
		\end{equation}
		If $v_{0}'\left(\xi\right)\ge0$, by \eqref{blow-up of general pressure}, we have $\left(J^{\varepsilon}\right)^{-1}v_{y}^{\varepsilon}\left(y_{1}^{\varepsilon}\left(\tau;\xi,0\right),\tau\right)\ge 0$ for all $\tau\geq0$. If $v_{0}'\left(\xi\right)<0$, we need to estimate $\int_{0}^{\tau}I\left(g^{\varepsilon}\left(y_{1}^{\varepsilon}\left(s;\xi,0\right),s\right)\right)ds$. 
		By \eqref{upper bound of p} and \eqref{the density lower bound of C1 initial value}, for $g^{\varepsilon}$, we have
		\begin{equation}\label{}
			\frac{A_{1}}{1+A_{2}B\tau} \le g^{\varepsilon}\left(y,\tau\right)
			\le \bar{g}^{\varepsilon}.
		\end{equation}
		Then, since $p\left(g^{\varepsilon}\right)\in C^{2}\left(\mathbb{R}^{+}\right)$ , for $g^{\varepsilon}\in \left[\frac{A_{1}}{1+A_{2}B\tau},\bar{g}^{\varepsilon}\right]$, by \eqref{condition of p}, a positive lower bound $\underline{I}^{\varepsilon}$ such that
		\begin{equation}
			I\left(g^{\varepsilon}\right) \ge \underline{I}^{\varepsilon} >0.
		\end{equation}
		Then we have the following inequality:
		\begin{equation}
			\int_{0}^{\tau}I\left(g^{\varepsilon}\left(y_{1}^{\varepsilon}\left(s;\xi,0\right),s\right)\right)ds \ge
			\int_{0}^{\tau}\underline{I}^{\varepsilon} ds.
		\end{equation}
		Thus, if $v_{0}'\left(\xi\right)<0$, when $\tau$ increases from 0 to some $T_{b}^{\varepsilon}\left(\xi\right)$, we have
		\begin{equation}
			1 + \left(g_{0}{p}'\left(g_{0}\right)v_{0}'\right) \left(\xi\right)\int_{0}^{T_{b}^{\varepsilon}\left(\xi\right)}I\left(g^{\varepsilon}\left(y_{1}^{\varepsilon}\left(s;\xi,0\right),s\right)\right)ds=0,
		\end{equation}
		which indicate $\left(J^{\varepsilon}\right)^{-1}v_{y}^{\varepsilon}\to -\infty$ as $\tau\to T_{b}^{\varepsilon}(\xi)$.
		
		For $\xi$ run over all the points satisfying $v_{0}'\left(\xi\right)<0$, we could have the minimum life-span
		\begin{equation}
			T_{b}^{\varepsilon}:=\inf \left \{ T_{b}^{\varepsilon }\left ( \xi \right ) \right \} >0.
		\end{equation}
		
		In Lagrangian coordinates, for fixed $\varepsilon>0$, $g^{\varepsilon}$ is upper and lower bounded in $\left[ 0,T_{b}^{\varepsilon} \right)$, so the global solution exists if and only if, for $y\in \mathbb{R}$
		\begin{equation}
			v_{0}'\left(y\right)\ge 0.
		\end{equation}
		On the other side, if $v_{0}'\left(y\right)<0$, $\left(J^{\varepsilon}\right)^{-1}v_{y}^{\varepsilon}$ will goes to $-\infty$ in the finite time. Then, there exists at least one point $\left(Y_{b}^{\varepsilon},T_{b}^{\varepsilon}\right)$ such that on $\mathbb{R}\times \left[0,T_{b}^{\varepsilon}\right)$, as $\left(y,\tau\right)\to\left(Y_{b}^{\varepsilon},T_{b}^{\varepsilon}\right)$,
		\begin{equation}
			\left(J^{\varepsilon}\right)^{-1}v_{y}^{\varepsilon} \to -\infty.
		\end{equation}

		Next, we want to discuss the transformation of the solution between the Lagrangian coordinates and the Eulerian coordinates. For $y\in\left[-L,L\right]$, 
		there is a characteristic triangle 
		\begin{equation}
			\left\{\left(y,\tau\right)\mid -L\le y\le y^{\varepsilon}_{1}\left(s;L,0\right), 0\le s\le \tau \right\},
		\end{equation}
		where $y^{\varepsilon}_{1}\left(\tau;L,0\right) = -L$.
		And, in the characteristic triangle, each $(y,\tau)$ has a unique $(x,t)$ satisfying
		\begin{equation} \label{relationship between x and y}
			\begin{cases}
				\frac{d}{dt}x\left(t; y\right) = v^{\varepsilon}\left(y,t\right), \\
				x\left(0;y\right)=y.
			\end{cases}
		\end{equation}
		which could be expressed as
		\begin{equation}
			x(y,\tau) = y + \int_{0}^{\tau}v^{\varepsilon}(y,s)ds.
		\end{equation}
		We let $L\to+\infty$, for any $y\in\mathbb{R}$, each $(y,\tau)$ has a unique $(x,t)$ by \eqref{relationship between x and y}. According to the transformation of Eulerian coordinates and Lagrangian coordinates, we have
		\begin{equation} \label{u-transformation}
			\begin{cases}
				u_{x}^{\varepsilon} = \left(J^{\varepsilon}\right)^{-1}v_{y}^{\varepsilon} ,\\
				u_{t}^{\varepsilon} = -\left(v^{\varepsilon} -\varepsilon^{2}g^{\varepsilon}{p}' \left(g^{\varepsilon}\right)\right)\left(J^{\varepsilon}\right)^{-1} v_{y}^{\varepsilon}.
			\end{cases}
		\end{equation}
		And, for $\rho^{\varepsilon}$ we have
		\begin{equation}\label{rho-transformation}
			\begin{cases}
				\rho_{x}^{\varepsilon} = \left(J^{\varepsilon}\right)^{-1}g_{y}^{\varepsilon}
				= \left(J^{\varepsilon}\right)^{-1}\frac{v'_{0}+\varepsilon^{2}p'\left(g_{0}\right)g'_{0}-v_{y}^{\varepsilon}}{\varepsilon^{2}p'\left(g^{\varepsilon}\right)},\\
				\rho_{t}^{\varepsilon} = -v^{\varepsilon}\left(J^{\varepsilon}\right)^{-1}g_{y}^{\varepsilon} - g^{\varepsilon} \left(J^{\varepsilon}\right)^{-1} v_{y}^{\varepsilon}.
			\end{cases}
		\end{equation}
		
		In Eulerian coordinates, for fixed $\varepsilon>0$, $u^{\varepsilon}$ and $\rho^{\varepsilon}$ is upper and lower bounded in $\left[ 0,T_{b}^{\varepsilon} \right)$, so the global solution exists if and only if, for $x\in \mathbb{R}$
		\begin{equation}
			u_{0}'\left(x\right)\ge 0.
		\end{equation}
		On the contrary, if $u_{0}'\left(x\right)<0$, $u_{x}^{\varepsilon}$ will goes to $-\infty$ in the finite time. Then, respecting to $\left(Y_{b}^{\varepsilon},T_{b}^{\varepsilon}\right)$, there exists  $\left(X_{b}^{\varepsilon},T_{b}^{\varepsilon}\right)$ such that on $\mathbb{R}\times \left[0,T_{b}^{\varepsilon}\right)$, as $\left(x,t\right)\to\left(X_{b}^{\varepsilon},T_{b}^{\varepsilon}\right)$,  $u_{x}^{\varepsilon}\left(x,t\right) \to -\infty$.
		
		Further, modulus of continuity estimates of the solution see \cite{Li Ta-Tsien Yu Wenci}.
	\end{proof}

	\section{The well-posedness contact discontinuous solution}
	\indent
	In this section we consider the equation $\eqref{AR}$ with piecewise Lipshcitzs initial data \eqref{new initial value}, where $v_0\in Lip\left(\mathbb{R}\right)$ and $Z_{0} \in \Lip_s\left(\mathbb{R}\right)$. 
	Without loss of generality, we can assume that $Z_0$
	only has a jump discontinuity at $y=0$.

	\subsection{The lower bound of density }
	First, we would to clarify the influence of the jump.  Here, we denote the characteristic line starting from $y=0$ as $y_1^\varepsilon(\tau)$, which is defined in \eqref{characteristic line}. And the characteristic line leading back from $\left(\tilde{y},\tilde{\tau}\right)$ to the initial data is denoted as $y_{1}^{\varepsilon}\left(s;\tilde{y},\tilde{\tau}\right)$. 
	For whether the characteristic line $y_{1}^{\varepsilon}\left(s;\tilde{y},\tilde{\tau}\right)$ crosses the discontinuity line or not, we can divide $\Omega:=\mathbb{R}\times\left[0,+\infty\right)$ into the following regions: 
	\begin{equation}
		\Omega = \Omega_{+}\cup \Omega^{\varepsilon}_{\rm{I}}\cup\Omega^{\varepsilon}_{\amalg }\cup\left\{y=0\right\},
	\end{equation}
	where
	\begin{equation*}
		\Omega_{+}=\left\{\left(y,\tau\right)\mid y>0, \tau\ge0\right\},\,\,	\Omega^{\varepsilon}_{\rm{I}}=\left\{\left(y,\tau\right)\mid y_{1}^{\varepsilon}\left(\tau\right)\le y < 0, \tau\ge0\right\}, \,\,\mbox{and}\,\, \Omega^{\varepsilon}_{\amalg}= \left\{\left(y,\tau\right)\mid y<y_{1}^{\varepsilon}\left(\tau\right), \tau\ge0\right\}.
	\end{equation*}
	For $\left(\tilde{y},\tilde{\tau}\right) \in \Omega_{+}\cup\Omega^{\varepsilon}_{\amalg }$, $y_{1}^{\varepsilon}\left(s;\tilde{y},\tilde{\tau}\right)$ does not cross the discontinuity line. 
	From the proof of Proposition 2.4, we can obtain the local existence of Lipschitz solutions in the characteristic triangle of domain $\Omega_{+}\cup \Omega^{\varepsilon}_{\amalg }$. For $\left(\tilde{y},\tilde{\tau}\right) \in \Omega^{\varepsilon}_{\rm{I}}$, its backward characteristic line $y_{1}^{\varepsilon}\left(s;\tilde{y},\tilde{\tau}\right)$ must reach $\left(0,\tau_{0}\right)$. So for $v\left(\tilde{y},\tilde{\tau}\right)$ and $Z^{\varepsilon}\left(\tilde{y},\tilde{\tau}\right)$ we have
	\begin{equation}
		v\left(\tilde{y},\tilde{\tau}\right) = v\left(0,\tau_{0}\right), \quad Z^{\varepsilon}\left(\tilde{y},\tilde{\tau}\right) = Z_{0}^{\varepsilon}\left(\tilde{y}\right).
	\end{equation}
	Therefore, if $g>0$, then $g$ in $\Omega^{\varepsilon}_{\rm{I}}$ can be expressed as 
	\begin{equation}
		g^{\varepsilon}\left(\tilde{y},\tilde{\tau}\right) = p^{-1}\left(\frac{Z_{0}^{\varepsilon}\left(\tilde{y}\right) - v\left(0,\tau_{0}\right) }{\varepsilon^{2}}\right).
	\end{equation}
	
	Without loss of generality, for different jump cases where $Z_{0}^{\varepsilon}$ has only one discontinuity at $y=0$, we discuss the lower bound estimate of density for different regions.
	The key point is based on the results of Proposition \ref{lower bound of C1 initial value}, we have the following proposition.
	\begin{proposition}\label{lower bound of rho in Lip case}
		For fixed $\varepsilon>0$ and the solutions $\left(J^{\varepsilon},v^{\varepsilon}, Z^{\varepsilon}\right)$ of \eqref{equation of Lagrangian transformation 2}-\eqref{new initial value}: 
		
		\noindent 		Case 1: $\lim_{g \to 0}p\left(g\right) = 0$ and the initial data satisfies the $\varepsilon$-condition;
		
		\noindent 		Case 2: $\lim_{g \to 0}p\left(g\right) = -\infty$ ( $\varepsilon$-condition is not required );
		
		\noindent 		then $g^{\varepsilon}$ has a lower bound with respect to $\varepsilon$.
	\end{proposition}
	\begin{proof}
		1. Region $\Omega_{+}\cup \Omega^{\varepsilon}_{\amalg }$: For any point $\left(\tilde{y},\tilde{\tau}\right)$ in $\Omega_{+}\cup\Omega^{\varepsilon}_{\amalg }$, by the method in Proposition \ref{lower bound of C1 initial value}, we can get the uniform lower bound estimate of density: 
		\begin{equation}\label{lower bound in region 1}
			g^{\varepsilon}\left(y,\tau\right)\ge \frac{A_{1}}{1+A_{2}B\tau}.
		\end{equation}
		where $A_{1}=\min_{y\in\mathbb{R}} g_{0}$, $A_{2}=\max_{y\in\mathbb{R}} g_{0}$, and $B$ can be defined as
		\begin{equation}
			B = \frac{\left(\Lip\left(Z_{0}^{\varepsilon}\right)\right)_{+}}{g_{0}},
		\end{equation}
		here $\left(Z_{0}^{\varepsilon}\right)_{C}=Z_0^\varepsilon-\left(Z_{0}^{\varepsilon}\right)_{J}$ is the absolute continuous part of $Z_0$, and $\left(\Lip(Z^\varepsilon_0)\right)_{+}$ is the Lipschitzs constant of $\left(Z_{0}^{\varepsilon}\right)_{C}$ without decreasing.
		
		2. The discontinuous curve $y=0$: If $g_{0}(0^{-})>g_{0}(0^{+})$ at $y=0$, since $v_{0}$ is continuous, this means $Z_{0}(0^{-})>Z_{0}(0^{+})$. By \eqref{D(1/g)}, similar to the case in $\Omega_{+}\cup \Omega^{\varepsilon}_{\amalg }$, there exists a constant $B$ such that the density has a lower bound \eqref{lower bound in region 1}.
		
		
		Otherwise, if $g_{0}(0^{-})<g_{0}(0^{+})$, which means $Z_{0}(0^{-})<Z_{0}(0^{+})$, for $(0, \tilde{\tau})$ and 
		$p\left(g^{\varepsilon}\left(0^{-},\tilde{\tau}\right)\right)$, we have
		\begin{equation}\label{The discontinuous curve case}
			p\left(g^{\varepsilon}\left(0^{-},\tilde{\tau}\right)\right) = \frac{Z^{\varepsilon}_{0}\left(0^{-}\right)-v_{0}\left(y_{1}^{\varepsilon}\left(0;0,\tilde{\tau}\right)\right)}{\varepsilon^{2}}.
		\end{equation}
		For \eqref{The discontinuous curve case}, we need to discuss the following two cases: 
		
		Case 1: $\lim_{g \to 0}p\left(g\right) = 0$ and the initial data satisfies the $\varepsilon$-condition. By \eqref{assumption1}, there exists a constant $A^{\varepsilon}$ which may depend on $\varepsilon$ such that
		\begin{equation}\label{lower bound on discontinuous line}
			p\left(g^{\varepsilon}\left(0^{-},\tilde{\tau}\right)\right) \ge A^{\varepsilon}> 0.
		\end{equation}
		So we have
		\begin{equation}\label{lower bound of rho}
			g^{\varepsilon}\left(0^{+},\tilde{\tau}\right) > g^{\varepsilon}\left(0^{-},\tilde{\tau}\right) \ge p^{-1}\left(A^{\varepsilon}\right) >0.
		\end{equation}
		
		Case 2: $\lim_{g \to 0}p\left(g\right) = -\infty$. By \eqref{The discontinuous curve case}, for $p(g)$ we have
		
		\begin{equation}\label{}
			p\left(g^{\varepsilon}\left(0^{-},\tilde{\tau}\right)\right) = p\left(g_{0}(0^{-})\right) + \frac{v_{0}(0)-v_{0}\left(y_{1}^{\varepsilon}\left(0;0,\tilde{\tau}\right)\right)}{\varepsilon^{2}}
			\ge p\left(\underline{g_{0}}\right) -  \frac{2 \max_{y\in\mathbb{R}} \left |  v_{0} (y)\right | }{\varepsilon^{2}} > - \infty.
		\end{equation}
		By \eqref{condition of p}, $g^{\varepsilon}\left(0^{-},\tilde{\tau}\right)$ has a lower bound with respect to $\varepsilon$.

		3. Region $\Omega^{\varepsilon}_{\rm{I}}$: For any point $\left(\tilde{y},\tilde{\tau}\right) \in \Omega^{\varepsilon
		}_{\rm{I}}$, its backward characteristic line must reach $\left(0,\tau_{0}\right)$.
		\begin{equation}
			y_{1}^{\varepsilon} \left( \tau_{0};\tilde{y},\tilde{\tau}\right) = 0.
		\end{equation}
		by \eqref{D(1/g)} we have
		\begin{equation}\label{}
			D\left(\frac{1}{g^{\varepsilon}}\right) = \frac{\left(Z_{0}^{\varepsilon}\right)_{y}}{g_{0}}.
		\end{equation}
		Integrating $s$ from $\tau_{0}$ to $\tilde{\tau}$ along the characteristic line $y_{1}^{\varepsilon}\left(\tau;\tilde{y},\tilde{\tau}\right)$, there exists a constant $B$ such that 
		\begin{equation}
			B = \frac{\left(\Lip\left(Z_{0}^{\varepsilon}\right)\right)_{+}}{g_{0}}.
		\end{equation}
		such that
		\begin{equation}
			\frac{1}{g^{\varepsilon}\left(\tilde{y},\tilde{\tau}\right)} -
			\frac{1}{\min_{0\leq\tau_0\leq\tilde{\tau}} g^{\varepsilon}\left(0^{-},\tau_{0}\right)} \le
			\frac{1}{g^{\varepsilon}\left(\tilde{y},\tilde{\tau}\right)} -
			\frac{1}{g^{\varepsilon}\left(0^{-},\tau_{0}\right)}
			= \int_{\tau_{0}}^{\tilde{\tau}} \frac{\left(\Lip\left(Z_{0}^{\varepsilon}\right)\right)_{+}}{g_{0}}ds \le B\left(\tilde{\tau}-\tau_{0}\right).
		\end{equation}
		So for fixed $\varepsilon$, we have
		\begin{equation}
			g^{\varepsilon}\left(\tilde{y},\tau_{0}\right) \ge \frac{A^{\varepsilon}_{1}}{1+A^{\varepsilon}_{1}B\tau}.
		\end{equation}
		where $A^{\varepsilon}_{1}=\min_{0\leq\tau_0\leq\tilde{\tau}} g^{\varepsilon}\left(0^{-},\tau_{0}\right)$, $B =\frac{\left(\Lip\left(Z_{0}^{\varepsilon}\right)\right)_{+}}{g_{0}}$.
		Combining the three cases, we obtain that the density has a uniform lower bound in $\Omega$ under the Lagrangian coordinate, denoted by $\underline{g}^{\varepsilon}$.
		\begin{equation}
			\underline{g}^{\varepsilon}: = \min\left\{\frac{A_{1}}{1+A_{2}B\tau}, p^{-1}\left(A^{\varepsilon}\right), \frac{A^{\varepsilon}_{1}}{1+A^{\varepsilon}_{1}B\tau}\right\}.
		\end{equation}

	\end{proof}
	\begin{remark}
		For $\lim_{g \to 0}p\left(g\right) = 0$, if the assumption in \eqref{assumption1} becomes an equality, there exists at least one point $y_{0}>0$ such that
		\begin{equation}
			Z^{\varepsilon}_{0}\left(0^{-}\right) = v_{0}\left(y_{0}\right).
		\end{equation}
		Then there exists a time $T^{\varepsilon}_{y_{0}}$ such that
		\begin{equation}
			y_{1}^{\varepsilon}\left(T^{\varepsilon}_{y_{0}};y_{0},0\right) = 0.
		\end{equation}
		So we have
		\begin{equation}
			p\left(g^{\varepsilon}\left(0^{-},T^{\varepsilon}_{y_{0}}\right)\right) =
			\frac{Z^{\varepsilon}\left(0^{-},T^{\varepsilon}_{y_{0}}\right) - v\left(y_{1}^{\varepsilon}\left(T^{\varepsilon}_{y_{0}};y_{0},0\right),T^{\varepsilon}_{y_{0}}\right)}{\varepsilon^{2}} = \frac{ Z^{\varepsilon}_{0}\left(0^{-}\right)-v_{0}\left(y_{0}\right) }{\varepsilon^{2}} = 0.
		\end{equation}
		This means there exists a finite time $T^{\varepsilon}_{y_{0}}$ such that
		\begin{equation}
			g^{\varepsilon} \left(0^{-},T^{\varepsilon}_{y_{0}}\right)= 0.
		\end{equation}
		Based on the above discussion, \eqref{assumption1} is a necessary and sufficient condition for the absence of vacuum.
	\end{remark}
	
	\begin{remark}
		For the case of $\lim_{g \to 0}p\left(g\right) = -\infty$, we do not need to use $\varepsilon$-condition to get the lower bound of density.
	\end{remark}

	\subsection{The well-posedness of Aw-Rascle model}
	Based on the above discussion, we prove Theorem \ref{theorem 1}.

	\begin{proof}

		By \eqref{blow-up of general pressure}, in Eulerian coordinates, we have
		\begin{equation}\label{blow up of Lip case}
			u_{x}^{\varepsilon}\left(x_{1}^{\varepsilon}\left(t;\eta ,0\right),t\right) \\
			= \frac{1}{\rho^{\varepsilon}p'\left(\rho^{\varepsilon}\right)\left(x_{1}^{\varepsilon}\left(t;\eta,0\right),t\right)} \cdot \frac{\left(\rho_{0} p'\left(\rho_{0}\right) u_{0}'\right) \left(\eta\right)}
			{1+\left(\rho_{0}p'\left(\rho_{0}\right) u_{0}'\right)\left(\eta\right)\int_{0}^{t}I \left(\rho^{\varepsilon}\left(x_{1}^{\varepsilon}\left(s;\eta,0\right),s\right)\right)ds}.
		\end{equation}
		If $u'_{0}\left(\eta\right)\ge0$, by \eqref{blow up of Lip case}, we have $u_{x}^{\varepsilon}\left(x_{1}^{\varepsilon}\left(t;\eta,0\right),t\right)\ge 0$ for all $t\geq0$. If $u_{0}'\left(\eta\right)<0$, we need to estimate $\int_{0}^{\tau}I\left(\rho^{\varepsilon}\left(x_{1}^{\varepsilon}\left(s;\eta,0\right),s\right)\right)ds$. 
		By Proposition \ref{lower bound of rho in Lip case} and \eqref{upper bound of p}, for $\rho^{\varepsilon}$, we have
		\begin{equation}\label{}
			\underline{\rho}^{\varepsilon} \le \rho^{\varepsilon}\left(x_{1}^{\varepsilon}\left(t;\eta,0\right),t\right)
			\le \bar{\rho}^{\varepsilon}.
		\end{equation}
		Then, since $p\left(\rho^{\varepsilon}\right)\in C^{2}\left(\mathbb{R}^{+}\right)$ , for $\rho^{\varepsilon}\in \left[\underline{\rho}^{\varepsilon},\bar{\rho}^{\varepsilon}\right]$, $I\left(\rho^{\varepsilon}\right)$ has a lower bound $\underline{I}^{\varepsilon}$ such that
		\begin{equation}
			I\left(\rho^{\varepsilon}\right) \ge \underline{I}^{\varepsilon} .
		\end{equation}
		Then we have the following inequality:
		\begin{equation}
			\int_{0}^{t}I\left(\rho^{\varepsilon}\left(x_{1}^{\varepsilon}\left(s;\eta,0\right),s\right)\right)ds \ge
			\int_{0}^{t}\underline{I}^{\varepsilon} ds.
		\end{equation}
		
		Thus, if $u_{0}'\left(\eta\right)<0$, when $t$ increases from 0 to some $T_{b}^{\varepsilon}\left(\eta\right)$, we have
		\begin{equation}
			1 + \left(\rho_{0}{p}'\left(\rho_{0}\right) u_{0}'\right) \left(\eta\right)\int_{0}^{T_{b}^{\varepsilon}\left(\eta\right)}I\left(\rho^{\varepsilon}\left(x_{1}^{\varepsilon}\left(s;\eta,0\right),s\right)\right)ds=0,
		\end{equation}
		which indicate $\left(J^{\varepsilon}\right)^{-1}v_{y}^{\varepsilon}\to -\infty$ as $\tau\to T_{b}^{\varepsilon}(\eta)$.
		
		To consider $\eta\in\mathbb{R}$, we could introduce the minimum life-span
		\begin{equation}
			T_{b}^{\varepsilon}:=\inf \left \{ T_{b}^{\varepsilon }\left ( \eta \right ) \right \} >0,
		\end{equation}
		where $\eta$ run over all the points satisfying $u_{0}'\left(\eta\right)<0$.
		In Eulerian coordinates, for fixed $\varepsilon>0$, $\rho^{\varepsilon}$ is upper and lower bounded in $\left[ 0,T_{b}^{\varepsilon} \right)$, so the global solution exists if and only if, for $x\in \mathbb{R}$
		\begin{equation}
			u_{0}'\left(x\right)\ge 0.
		\end{equation}
		On the other side, if $u_{0}'\left(x\right)<0$, $u_{x}^{\varepsilon}$ will goes to $-\infty$ in the finite time. Then,  there exists  $\left(X_{b}^{\varepsilon},T_{b}^{\varepsilon}\right)$ such that on $\mathbb{R}\times \left[0,T_{b}^{\varepsilon}\right)$, as $\left(x,t\right)\to\left(X_{b}^{\varepsilon},T_{b}^{\varepsilon}\right)$, $u_{x}^{\varepsilon}\left(x,t\right) \to -\infty$. Therefore, by \eqref{u-transformation} and \eqref{rho-transformation}, $\rho^{\varepsilon}$ and $u^{\varepsilon}$ satisfy
		
		\begin{equation}
			\rho^{\varepsilon}\in \Lip_s\left(\mathbb{R}\times\left[0,T_{b}^{\varepsilon}\right)\right),\quad u^{\varepsilon}\in  \Lip \left(\mathbb{R}\times\left[0,T_{b}^{\varepsilon}\right)\right).
		\end{equation}

		Next, we consider the regularity of discontinuous line. The discontinuous line in Eulerian coordinates is
		\begin{equation}
			\frac{dx_{2}^{\varepsilon}\left(t\right)}{dt}=u^{\varepsilon}\left(x_{2}^{\varepsilon}\left(t\right),t\right),
		\end{equation}
		which has the following implicit expression
		\begin{equation}\label{discontinuous line}
			x_{2}^{\varepsilon}\left(t\right) = x_{0} + \int_{0}^{t} u^{\varepsilon} \left(x_{2}^{\varepsilon}\left(s\right),s\right) ds.
		\end{equation}
		Therefore, by \eqref{discontinuous line}, for $h>0$, $t\ge 0$, we have
		\begin{eqnarray}
			\left | \left(x_{2}^{\varepsilon}\right)'\left(t+h\right) -\left(x_{2}^{\varepsilon}\right)'\left(t\right)\right | & = & \left | u^{\varepsilon}\left(x_{2}^{\varepsilon}\left(t+h\right),t+h\right) - u^{\varepsilon}\left(x_{2}^{\varepsilon}\left(t\right),t\right)\right | \nonumber \\
			& \le & \Lip(u^{\varepsilon})\left(\left | x_{2}^{\varepsilon}\left(t+h\right) - x_{2}^{\varepsilon}\left(t\right) \right | + h \right) \nonumber \\
			& \le & \Lip(u^{\varepsilon})\left( \left \| u \right \| _{L^{\infty}} + 1 \right)h.
		\end{eqnarray}
		So $\left(x^{\varepsilon}\right)'$ is a Lipschitz function with respect to $t$.

		
	\end{proof}
	
	\begin{remark}
		We have proved the well-posedness of the contacted discontinuity solution where $\rho_0$ has one contact discontinuity. For a general piecewise Lipschitz functions function $\rho_0$, we can follow the above idea. Since the discontinuity points where $[\rho_0]>0$ are separable, we can repeat the above procedure for each characteristic triangle containing only one such discontinuity. Then we can glue the fragments together to establish local existence. By repeating this procedure, we can extend the solution's existence time forward until a nonlinear singularity appears. Therefore we proved the well-posedness for general piecewise Lipschitz functions $\rho_0$.
	\end{remark}

	For fixed $\varepsilon>0$, if the initial data are $\rho_{0},u_{0}\in C^{1}$, by the method in \cite{Li Ta-Tsien Yu Wenci}, for a point $\left(x,t\right)$ in the domain $\mathbb{R}\times\left[0,T_{b}^{\varepsilon}\right)$, we can obtain the uniform modulus of continuity estimation of $u_{x}$ and $u_{t}$, and then obtain the existence of $C^{1}$ solution. 
	However, for the case where there is a discontinuity, since the backward characteristic line of the point in the region $\Omega_{+}\cup \Omega^{\varepsilon}_{\amalg }$ does not pass through the discontinuity, the regularity of the solution in the region can be improved to $C^{1}$ by the uniform modulus of continuity estimation. However, the backward characteristic line of the points in the region $\Omega^{\varepsilon}_{\rm{I}}$ will pass through the discontinuity, we cannot obtain the uniform modulus of continuity estimation of $u_{x}$ and $u_{t}$ along the spatial direction, so the Lipschitz regularity is optimal.

	\subsection{Pressureless fluid case}
	The blow-up of pressureless fluid in Eulerian coordinates can refer to \cite{Peng Wenjian Wang Tian-Yi}, and we have similar argument in Lagrangian coordinates.
	For a smooth solution $\left(\bar{\rho},\bar{u}\right)$ of \eqref{pressureless Euler}, $\eqref{pressureless Euler}_{2}$ is equivalent to Burgers’ equation
	\begin{equation}
		\bar{u}_{t} + \bar{u}\bar{u}_{x} = 0.
	\end{equation}
	For pressureless fluid, we introduce the following derivative
	\begin{equation}
		D_{0} = \partial_{t} + \bar{u} \partial_{x}.
	\end{equation}
	And the characteristic line passing through $\left(\tilde{x},\tilde{t}\right)$ is defined as
	\begin{equation} \label{bar x}
		\begin{cases}
			\frac{d\bar{x}\left(t;\tilde{x},\tilde{t}\right)}{dt} = \bar{u} \left(\bar{x}\left(t;\tilde{x},\tilde{t}\right),t\right) , \\
			\bar{x}\left(\tilde{t};\tilde{x},\tilde{t}\right) = \tilde{x}.
		\end{cases}
	\end{equation}

	Under the Lagrangian transformation, the equation \eqref{pressureless Euler} is equivalent to
	\begin{equation}\label{Lagrangian transformation for pressureless}
		\begin{cases}
			\bar{J}_{\tau} = \bar{v}_{y},\\
			\bar{v}_{\tau}=0.
		\end{cases}
	\end{equation}
	and the respective initial data are
	\begin{equation}\label{new initial value for pressureless}
		\left ( \bar{J}, \bar{v} \right ) |_{\tau=0} =\left ( 1, v_{0}\left(y\right) \right ).
	\end{equation} 
	
	Next, for the smooth solution $\left(\bar{J},\bar{v}\right)$ of \eqref{Lagrangian transformation for pressureless}.
	By $\eqref{Lagrangian transformation for pressureless}_{2}$ we have
	\begin{equation}\label{bar v and v_y}
		\bar{v}\left(y,\tau\right) = v_{0}\left(y\right) \,\,\,{\rm and}\,\,\,\bar{v}_y(y,\tau)=v_0'(y).
	\end{equation}
	Integrating the first equation in \eqref{Lagrangian transformation for pressureless}, we have
	\begin{equation}\label{bar J}
		\bar{J} \left(y,\tau\right)=1+v_{0}'\left(y\right)\tau.
	\end{equation}
	Combining \eqref{J},
	\begin{equation}\label{bar rho}
		\bar{g} \left(y,\tau\right)=\frac{g_0(y)}{J(y,\tau)}=\frac{g_{0} \left(y\right)}{1+v_{0}'\left(y\right)\tau}.
	\end{equation}
	Then, by \eqref{bar v and v_y} and \eqref{bar J}, one could get
	\begin{equation}\label{bar v_y}
		\bar{J}^{-1}\bar{v}_{y}\left(y,\tau\right)= \frac{v_{0}'\left(y\right)}{1+v_{0}'\left(y\right)\tau}.
	\end{equation}
	According to \eqref{bar v_y}, 
	if $v_{0}'\left(y\right)<0$, when $\tau$ increases from 0 to some $T_{b}\left(y\right)$, such that
	\begin{equation}
		1+v_{0}'\left(y\right)T_{b}\left(y\right)=0,
	\end{equation}
	which indicate $\bar{J}^{-1}\bar{v}_{y}\to -\infty$ and $\bar{g}\to+\infty$ as $\tau\to T_{b}(y)$.
	
	For $y$ run over all the points such that $v_{0}'\left(y\right)<0$, we could have the minimum life-span
	\begin{equation}\label{Tb}
		T_{b}:=\inf \left \{ T_{b}\left ( y \right ) \right \} = \inf \left \{ -\frac{1}{v'_{0}\left(y\right)} \right \} >0.
	\end{equation}
	
	Thus, from \eqref{bar v_y}, we see that singularity of $\bar{v}_{y}$ first happens at $\tau=T_{b}$. When $\tau \to T_{b}$, there is a point such that $\bar{g}$ goes to $+\infty$, which corresponds to the mass concentration.
	
	According to the transformation of Eulerian coordinates and Lagrangian coordinates, we have
	\begin{equation} \label{}
		\begin{cases}
			\bar{u}_{x} = \bar{J}^{-1}\bar{v}_{y} ,\\
			\bar{u}_{t} = -\bar{v} \bar{J}^{-1}\bar{v}_{y}.
		\end{cases}
	\end{equation}
	Therefore, in Eulerian coordinates, we have
	\begin{equation}
		\bar{\rho} \left(\bar{x}\left(t;x,0\right),t\right)=\frac{\rho_{0} \left(x\right)}{1+ u_{0}'\left(x\right)t},
	\end{equation}
	and
	\begin{equation}\label{}
		\bar{u}_{x}\left(\bar{x}\left(t;x,0\right),t\right)= \frac{u_{0}'\left(x\right)}{1+ u_{0}'\left(x\right)t}.
	\end{equation}
	According to the above discussion, for $t\in \left[0,T_{b}\right)$, if 
	\begin{equation}\label{initial value of pressureless fluid}
		\rho_{0}\in Lip_{s}\left(\mathbb{R}\right),\quad u_{0}\in Lip \left(\mathbb{R}\right),
	\end{equation}
	then $\bar{\rho}\in Lip_{s}\left(\mathbb{R}\times\left[0,T_{b}\right)\right)$, $\bar{u}\in Lip \left(\mathbb{R}\times\left[0,T_{b}\right)\right)$.

	Based on the above discussion, for the case of pressureless fluid, similar to Theorem \ref{theorem 1}, we have the following proposition. 
	\begin{proposition}
		For \eqref{pressureless Euler} with initial data \eqref{initial value of pressureless fluid}, there exists a time $T_{b}$ such that on $\mathbb{R}\times\left[0,T_{b}\right)$, there exists a solution $\left(\bar{\rho},\bar{u}\right)$ satisfying 
		
		(1) If $\inf_{x\in\mathbb{R}} u_{0}'(x)\ge 0$, the solution exists globally. For $\bar{\rho}$ and $\bar{u}$, there are
		\begin{equation}
			\bar{\rho}\in Lip_s \left(\mathbb{R}\times\left[0,+\infty\right) \right),\quad
			\bar{u}\in Lip \left(\mathbb{R}\times\left[0,+\infty\right)\right) .
		\end{equation}
		
		(2) If $\inf_{x\in\mathbb{R}} u_{0}'\left(x\right)<0$, there exists a finite $T_{b}$ and at least one $X_{b}$ such that as $\left(x,t\right)\to \left(X_{b},T_{b}\right)$, $\bar{u}_{x}\left(x,t\right)\to -\infty$ and $\bar{\rho}\left(x,t\right)\to +\infty$. And the solution stands
		\begin{equation}
			\bar{\rho}\in Lip_s \left(\mathbb{R}\times\left[0,T_{b}\right) \right),\quad 
			\bar{u}\in Lip \left(\mathbb{R}\times\left[0,T_{b}\right)\right).
		\end{equation}
		And, 	$\bar{\Gamma} = \left \{ \left ( x,t \right ) \in \mathbb{R}\times\left[0,T_{b}\right) : x=\bar{x}\left(t\right) \right \} $ is the discontinuous curve and $\bar{x}'\left(t\right)$ is a Lipschitz function with respect to $t$, satisfying
		\begin{equation}
			\frac{d\bar{x}\left(t\right)}{dt} = \bar{u}\left(\bar{x}\left(t\right),t\right) .
		\end{equation}
	\end{proposition}

	\section{On vanishing pressure limit}
	
	From the above discussion, we know that for Aw-Rascle model, if $u_{0}'<0$ initially, $u_{x}$ will goes to $-\infty$ in finite time. In this section, for any fixed $T$, as $\varepsilon \to 0$, we consider the convergence of $v^{\varepsilon}$ and $g^{\varepsilon}$ on $\mathbb{R}\times\left[0,T\right]$ in Lagrangian coordinates. Without loss of generality, we assume that $\varepsilon<1$. First, we introduce the level set on the lower bound of $\left(J^{\varepsilon}\right)^{-1}v_{y}^{\varepsilon}$:
	\begin{equation}
		m^{\varepsilon}\left(\tau\right):=\inf\limits_{y\in \mathbb{R} } \left \{ \left(J^{\varepsilon}\right)^{-1}v_{y}^{\varepsilon} \left ( y,\tau  \right ) \right \} .
	\end{equation}
	For any fixed $\varepsilon>0$ and the compressive initial data, there exists a finite life-span $T_{b}^{\varepsilon}$ defined in Proposition \ref{life-span of C^1 solution}, which is $+\infty$ for the rarefaction initial data. And we have
	\begin{equation}
		\lim_{\tau\uparrow T_{b}^{\varepsilon } } m^{\varepsilon} \left ( \tau  \right ) =-\infty .
	\end{equation}
	Further, for $M>0$, we define $\tau_{M}^{\varepsilon}$
	\begin{equation}
		\tau_{M}^{\varepsilon}=\sup\left\{s:-M\le \inf\limits_{\tau\in\left[0,s\right]}m^{\varepsilon}\left(\tau\right),s\le T\right\}.
	\end{equation}
	According to the definition, $\left\{\tau_{M}^{\varepsilon}\right\}$ is a monotone sequence with respect to $M$, and
	\begin{equation}\label{tau-M}
		\lim_{M \to +\infty } \tau _{M}^{\varepsilon} =\min_{ } \left \{ T_{b}^{\varepsilon },T\right \}=:\tau _{b}^{\varepsilon } .
	\end{equation}

	First, we need to proof that the density has a uniform lower bound with respect to $\varepsilon$, and we have the following proposition.
	\begin{proposition}\label{vanishing pressure limit lower bound}
		For any $\varepsilon>0$ and the solution $\left(J^{\varepsilon},v^{\varepsilon},Z^{\varepsilon}\right)$ of \eqref{equation of Lagrangian transformation 2}-\eqref{new initial value} with $0-$condition, 
		$g^{\varepsilon}$ has a uniform lower bound with respect to $\varepsilon$.
	\end{proposition}
	\begin{proof}
		1. Region $\Omega_{+}\cup \Omega^{\varepsilon}_{\amalg }$: For any point $\left(\tilde{y},\tilde{\tau}\right)$ in $\Omega_{+}\cup\Omega^{\varepsilon}_{\amalg }$, 
		by proposition \ref{lower bound of rho in Lip case}, the density has a uniform lower bound:
		\begin{equation}\label{}
			g^{\varepsilon}\left(y,\tau\right)\ge \frac{A_{1}}{1+A_{2}B\tau}.
		\end{equation}
		where $A_{1}=\min_{y\in\mathbb{R}} g_{0}$, $A_{2}=\max_{y\in\mathbb{R}} g_{0}$, and $B$ can be defined as
		\begin{equation}
			B=  \frac{\left(\Lip\left(Z_{0}^{\varepsilon}\right)\right)_{+}}{g_{0}},
		\end{equation}
		here $\left(Z_{0}^{\varepsilon}\right)_{C}=Z_0^\varepsilon-\left(Z_{0}^{\varepsilon}\right)_{J}$ is the absolute continuous part of $Z_0$, and $\left(\Lip\left(Z_{0}^{\varepsilon}\right)\right)_{+}$ 
		is the Lipschitzs constant of the continuous part of $Z^\varepsilon_0$ without decreasing
		
		2. The discontinuous curve $y=0$: If $g_{0}(0^{-})>g_{0}(0^{+})$ at $y=0$. Since $v_{0}$ is continuous, $Z_{0}(0^{-})>Z_{0}(0^{+})$. By \eqref{D(1/g)}, similar to the case in $\Omega_{+}\cup \Omega^{\varepsilon}_{\amalg }$, there exists a constant $B$ such that the density has a lower bound \eqref{lower bound in region 1}.

		If $g_{0}(0^{-})<g_{0}(0^{+})$ at $y=0$, for $(0, \tilde{\tau})$ 
		and $p\left(g^{\varepsilon}\left(0^{-},\tilde{\tau}\right)\right)$, we have
		\begin{equation}
			p\left(g^{\varepsilon}\left(0^{-},\tilde{\tau}\right)\right) = \frac{Z^{\varepsilon}_{0}\left(0^{-}\right)-v_{0}\left(y_{1}^{\varepsilon}\left(0;0,\tilde{\tau}\right)\right)}{\varepsilon^{2}}
			= p\left(g_{0}\left(0^{-}\right)\right) + \frac{v_{0}\left(0^{-}\right)-v_{0}\left(y_{1}^{\varepsilon}\left(0;0,\tilde{\tau}\right)\right)}{\varepsilon^{2}}.
		\end{equation}
		By \eqref{assumption2}, we have
		\begin{equation}\label{lower bound on discontinuous line}
			p\left(g^{\varepsilon}\left(0^{-},\tilde{\tau}\right)\right) \ge p\left(g_{0}\left(0^{-}\right)\right) \ge p\left(\underline{g_{0}}\right).
		\end{equation}
		So we have
		\begin{equation}\label{lower bound of rho}
			g^{\varepsilon}\left(0^{+},\tilde{\tau}\right) > g^{\varepsilon}\left(0^{-},\tilde{\tau}\right) \ge \underline{g_{0}}.
		\end{equation}
		
		3. Region $\Omega^{\varepsilon}_{\rm{I}}$: For any point $\left(\tilde{y},\tilde{\tau}\right) \in \Omega^{\varepsilon
		}_{\rm{I}}$, by proposition \ref{lower bound of rho in Lip case}, we have
		\begin{equation}
			g^{\varepsilon}\left(\tilde{y},\tilde{\tau}\right)
			\ge\frac{\min g^{\varepsilon}\left(0^{-},\tau_{0}\right) }{1+\min g^{\varepsilon}\left(0^{-},\tau_{0}\right)B\left(\tilde{\tau}-\tau_{0}\right)} \ge \frac{A'_{1}}{1+A'_{1}B\tau}.
		\end{equation}
		where $A'_{1}=\min g^{\varepsilon}\left(0^{-},\tau_{0}\right)$, $B =  \frac{\left(\Lip\left(Z_{0}^{\varepsilon}\right)\right)_{+}}{g_{0}}$.
		Combining the above cases, we obtain that the density has a lower bound in $\Omega$ under the Lagrangian coordinate, denoted by $\underline{g}$:
		\begin{equation}\label{The lower bound independent of epsilon}
			\underline{g}: = \min\left\{\frac{A_{1}}{1+A_{2}B\tau}, \underline{g_{0}}, \frac{A'_{1}}{1+A'_{1}B\tau}\right\}.
		\end{equation}
		
	\end{proof}
	
	\begin{remark}
		In order to get a uniform lower bound of $g^{\varepsilon}$, for constant $M$, we need
		\begin{equation}
			p\left(g^{\varepsilon}\left(y,\tau\right)\right) = p\left(g_{0}(y)\right) + \frac{v_{0}(y)-v_{0}\left(y_{1}^{\varepsilon}\left(0;y,\tau\right)\right)}{\varepsilon^{2}} \ge M >-\infty.
		\end{equation}
		The above equation is equivalent to
		\begin{equation}
			v_{0}(y) \ge v_{0}\left(y_{1}^{\varepsilon}\left(0;y,\tau\right)\right) + \varepsilon^{2}\left(M-p\left(g_{0}(y)\right)\right).
		\end{equation}
		As $\varepsilon\to 0 $, the above equation is equivalent to the $0$-condition, which is necessary and sufficient for both  $\lim_{g \to 0}p\left(g\right) = 0$ case and  $\lim_{g \to 0}p\left(g\right) = -\infty$ case.
	\end{remark}

	For fixed sufficiently large $M$, let $\underline{\tau}_{M}:=\varliminf_{\varepsilon\to0}\tau_{M}^{\varepsilon}$, see Lemma \ref{positive lower bound of t} for specific proof. Next, we have the following $L^{\infty}$ estimates lemma. 
	\begin{lemma}\label{uniform estimations} ($L^{\infty}$ uniform estimations) . On $\mathbb{R}\times \left[0,\underline{\tau}_{M}\right]$, for any $\varepsilon>0$,  $g^{\varepsilon}$ has uniform bound with respect to $\varepsilon$ in $L^{\infty}\left(\mathbb{R}\times \left[0,\underline{\tau}_{M}\right]\right)$; and $v^{\varepsilon}$ has uniform bound with respect to $\varepsilon$ in $Lip\left(\mathbb{R}\times \left[0,\underline{\tau}_{M}\right]\right)$ .
	\end{lemma}
	\begin{proof}
		On $\mathbb{R}\times \left[0,\underline{\tau}_{M}\right]$, we have
		\begin{equation}
			\ln_{}{\frac{g^{\varepsilon}\left(y,\tau\right)}{g_{0}\left(y\right)}} = \int_{0}^{\tau}\left(\ln_{}{g^{\varepsilon}\left(y,s\right)} \right)_{\tau}ds  =-\int_{0}^{\tau}\left(J^{\varepsilon}\right)^{-1}v_{y}^{\varepsilon}\left(y,s\right) ds \le MT.
		\end{equation}
		Therefore, on $\mathbb{R}\times \left[0,\underline{\tau}_{M}\right]$, let $\bar{g}_{0}:=\max_{y\in\mathbb{R}}g_{0}$ we obtain the uniform upper bound of $g^{\varepsilon}$.
		\begin{equation}
			g^{\varepsilon}\le \bar{g}_{0} e^{MT} =: \bar{g}_{M}.
		\end{equation}
		Combining \eqref{The lower bound independent of epsilon}, then $g^{\varepsilon}$ has uniform bound with respect to $\varepsilon$ on $\mathbb{R}\times \left[0,\underline{\tau}_{M}\right]$:
		\begin{equation}\label{the uniform estimation of J}
			\underline{g} \le g^{\varepsilon} \le \bar{g}_{M}.
		\end{equation}
		By \eqref{v=v0}, $v^{\varepsilon}$ has uniform bound with respect to $\varepsilon$
		\begin{equation}
			\min_{y\in\mathbb{R}} v_{0}(y) \le v^{\varepsilon}(y,\tau)\le \max_{y\in\mathbb{R}}v_{0}(y).
		\end{equation}
		Next, we estimate the bound of $\left(J^{\varepsilon}\right)^{-1}v_{y}^{\varepsilon}$, due to
		\begin{equation}
			D\left(v^{\varepsilon}_{\tau}\right)=-\frac{I\left(g^{\varepsilon}\right)}{\varepsilon^{2}}\left(v^{\varepsilon}_{\tau}\right)^{2}\le 0.
		\end{equation}
		So we have
		\begin{equation}
			\varepsilon^{2}\left(g^{\varepsilon} p'\left(g^{\varepsilon}\right) \left(J^{\varepsilon}\right)^{-1}v_{y}^{\varepsilon}\right)\left(y,\tau\right) \le \varepsilon^{2}\left(g_{0} p'\left(g_{0}\right)  v_{0}'\right)\left(y_{1}^{\varepsilon}\left(0;y,\tau\right)\right).
		\end{equation}
		$g^{\varepsilon}$ is uniformly bounded for any $\varepsilon$ and $p\in C^{2}\left(\mathbb{R^{+}}\right)$, so $ p'\left(g^{\varepsilon}\right) $ is uniformly bounded. So we obtain the uniform upper bound of $\left(J^{\varepsilon}\right)^{-1}v_{y}^{\varepsilon}$ on $\mathbb{R}\times \left[0,\underline{\tau}_{M}\right]$
		\begin{equation}\label{}
			\left(\left(J^{\varepsilon}\right)^{-1}v_{y}^{\varepsilon}\right)\left(y,\tau\right) \le \frac{\left(g_{0} p'\left(g_{0}\right)   v_{0}'\right)\left(y_{1}^{\varepsilon}\left(0;y,\tau\right)\right)}
			{ \left(g^{\varepsilon} p' \left(g^{\varepsilon}\right) \right) \left(y,\tau\right)}  \\
			\le C_{1}\left(M\right),
		\end{equation}
		On the other hand, due to
		\begin{equation}
			\left(J^{\varepsilon}\right)^{-1}v_{y}^{\varepsilon}\ge -M,
		\end{equation}
		on $\mathbb{R}\times \left[0,\underline{\tau}_{M}\right]$, we have
		\begin{equation}
			\left \| \left(J^{\varepsilon}\right)^{-1}v_{y}^{\varepsilon}\right \| _{\infty } \le \max_{} \left \{
			C_{1}\left(M\right) ,M \right \} \le C\left(M\right).
		\end{equation}
	\end{proof}
	
	According to the transformation of Eulerian coordinates and Lagrangian coordinates, we have
	\begin{equation} \label{}
		\begin{cases}
			u_{x}^{\varepsilon} = \left(J^{\varepsilon}\right)^{-1}v_{y}^{\varepsilon} ,\\
			u_{t}^{\varepsilon} = -\left(v^{\varepsilon} -\varepsilon^{2}g^{\varepsilon} p'\left(g^{\varepsilon}\right)  \right)\left(J^{\varepsilon}\right)^{-1}v_{y}^{\varepsilon}.
		\end{cases}
	\end{equation}
	
	Correspondingly we have the following lemma.
	\begin{lemma} \label{uniform estimations 2}
		On $\mathbb{R}\times \left[0,\underline{t}_{M}\right]$, for any $\varepsilon>0$, $\rho^{\varepsilon}$ has uniform bound with respect to $\varepsilon$ in $L^{\infty}\left(\mathbb{R}\times \left[0,\underline{t}_{M}\right]\right)$; and $u^{\varepsilon}$ has uniform bound with respect to $\varepsilon$ in $Lip\left(\mathbb{R}\times \left[0,\underline{t}_{M}\right]\right)$.
	\end{lemma}

	Before we discuss the convergence of $\rho^{\varepsilon}$ and $u^{\varepsilon}$,
	we need to estimate the uniform lower bound of $t_{M}^{\varepsilon}$ with respect to $\varepsilon$, which means that for $M$ large enough, $t_{M}^{\varepsilon}$ does not go to 0 as $\varepsilon$ goes to 0.
	
	\begin{lemma}\label{positive lower bound of t}
		For $M$ is large enough, $t_{M}^{\varepsilon}$ has a uniform positive lower bound.
	\end{lemma}
	\begin{proof}
		
		This claim is obviously true for $u_{0}'\ge0$, we then prove it by contradiction for $u_{0}'<0$.
		If the claim is false, we can find a subsequence $\left\{\varepsilon_{n}\right\}$ such that $t_{M}^{\varepsilon_{n}} \to 0 $ as $n\to \infty$. By \eqref{blow-up of general pressure}, in Eulerian coordinates, we can find a point $\left(x_{M}^{\varepsilon_{n}},t_{M}^{\varepsilon_{n}}\right)$ such that
		\begin{equation}\label{}
			-M = u_{x}^{\varepsilon_{n}}\left(x_{M}^{\varepsilon_{n}},t_{M}^{\varepsilon_{n}}\right) 
			= \frac{1}{\left(\rho^{\varepsilon_{n}}p'\left(\rho^{\varepsilon_{n}}\right)\right)\left(x_{M}^{\varepsilon_{n}},t_{M}^{\varepsilon_{n}}\right)} \cdot \frac{\left(\rho_{0} p'\left(\rho_{0}\right) u_{0}'\right) \left(x_{1}^{\varepsilon_{n}}\left(0;x_{M}^{\varepsilon_{n}},t_{M}^{\varepsilon_{n}}\right)\right)}
			{1+\left(\rho_{0}p'\left(\rho_{0}\right) u_{0}'\right)\left(x_{1}^{\varepsilon_{n}}\left(0;x_{M}^{\varepsilon_{n}},t_{M}^{\varepsilon_{n}}\right)\right)\int_{0}^{t_{M}^{\varepsilon_{n}}}I \left(\rho^{\varepsilon_{n}}\left(x_{M}^{\varepsilon_{n}},s\right)\right)ds}.
		\end{equation}
		For $\varepsilon_{n}>0$, $\rho^{\varepsilon_{n}}$ has uniform upper and lower bounds
		\begin{equation}\label{bound of rho 2}
			\underline{\rho} \le \rho^{\varepsilon_{n}}\le 
			\bar{\rho}_{0} e^{Mt_{M}^{\varepsilon_{n}}}.
		\end{equation}
		Therefore, for $\rho^{\varepsilon_{n}}\in\left[\underline{\rho}, \bar{\rho}_{0} e^{Mt_{M}^{\varepsilon_{n}}}\right]$, $I(\rho)$ is a continuous function respect to $\rho$, $I\left(\rho^{\varepsilon_{n}}\right)$ has a uniform upper and lower bounds.
		By condition \eqref{condition of p} satisfied by $p\left(\rho\right)$, we have
		\begin{equation}
			\left(\rho^{2}p'\left(\rho\right)\right)' = 2\rho p'\left(\rho\right) + \rho^{2} p''\left(\rho\right) >0.
		\end{equation}
		Combining the \eqref{bound of rho 2} and monotonicity of $\rho^{2} p'\left(\rho\right)$, we have
		\begin{equation}\label{bound of rhop'}
			\rho^{\varepsilon_{n}}p'\left(\rho^{\varepsilon_{n}}\right) = \frac{\left(\rho^{\varepsilon_{n}}\right)^{2}p'\left(\rho^{\varepsilon_{n}}\right)}{\rho^{\varepsilon_{n}}} \ge 
			\frac{\left(\underline{\rho}\right)^{2}p'\left(\underline{\rho}\right)}{\bar{\rho}_{0} e^{Mt_{M}^{\varepsilon_{n}}}}.
		\end{equation}
		So, as $n\to \infty$, 
		\begin{eqnarray}\label{by contradiction}
			-M 
			& \ge & \frac{\bar{\rho}_{0} e^{Mt_{M}^{\varepsilon_{n}}}}{\left(\underline{\rho}\right)^{2}p'\left(\underline{\rho}\right)} \cdot \frac{\min_{x\in \mathbb{R}}\left(\rho_{0} p'\left(\rho_{0}\right) u_{0}'\right) }
			{1+\left(\rho_{0}p'\left(\rho_{0}\right) u_{0}'\right)\left(x_{1}^{\varepsilon_{n}}\left(0;x_{M}^{\varepsilon_{n}},t_{M}^{\varepsilon_{n}}\right)\right)\int_{0}^{t_{M}^{\varepsilon_{n}}}I \left(\rho^{\varepsilon_{n}}\left(x_{M}^{\varepsilon_{n}},s\right)\right)ds} \nonumber \\
			& \to & 
			\frac{\bar{\rho}_{0} \cdot \min_{x\in \mathbb{R}}\left(\rho_{0} p'\left(\rho_{0}\right) u_{0}'\right)}{\left(\underline{\rho}\right)^{2}p'\left(\underline{\rho}\right)}.
		\end{eqnarray}
		On the other hand, when $M$ is sufficiently large
		\begin{equation}
			\frac{\bar{\rho}_{0} \cdot \min_{x\in \mathbb{R}}\left(\rho_{0} p'\left(\rho_{0}\right) u_{0}'\right)}{\left(\underline{\rho}\right)^{2}p'\left(\underline{\rho}\right)} \ge -M.
		\end{equation}
		This contradicts with \eqref{by contradiction}.
		Therefore, for each fix $M$ large enough, $t_{M}^{\varepsilon}$ has a uniformly positive lower bound, denoted by $\underline{T}$.

	\end{proof}

	Based on the uniform estimates discussed above, we then prove Theorem \ref{Vanishing Pressure Limit}(1). There is a uniform domain defined as $\underline{t}_{M}:=\varliminf_{\varepsilon\to0}t_{M}^{\varepsilon}$.
	For the compact set $H\subset \mathbb{R}\times\left[0,\underline{t}_{M}\right]$, the estimates we did above are uniform in $H$, we have
	\begin{equation}\label{convergence}
		u^{\varepsilon}\to \bar{u} \,\, {\rm in} \,\,  {\Lip} \left(H\right),\quad \rho^{\varepsilon}\to \bar{\rho} \,\, {\rm in} \,\, \mathcal{M}\left(H \right).
	\end{equation}
	As $\varepsilon\to 0$, combine the $\eqref{x-i}$ and the convergence of $(u^\varepsilon,\rho^\varepsilon)$, we have 
	\begin{equation}
		x^{\varepsilon}\left(t\right) \to \bar{x}\left(t\right).
	\end{equation}
	where $\bar{x}\left(t\right)$ is a Lipschitz function with respect to $t$. 
	\begin{remark}
		In Lagrangian coordinates, for $\tau\le\underline{\tau}_{M}$ and any point $\left(y,\tau\right)$ in $\Omega^{\varepsilon}_{\rm{I}}$, when $\varepsilon$ is sufficiently small, the backward characteristic line of this point is included in $\Omega_{-}$. That is, when $\varepsilon\to 0$, $\Omega^{\varepsilon}_{\amalg}\to \Omega_{-}$ and $\Omega=\bar{\Omega}_{+}\cup \Omega_{-}$.
	\end{remark}
	
	Next,  we consider the consistency. For $\varepsilon>0$ and $\varphi\in C_{c}^{\infty}\left(H\right)$, we have
	\begin{equation}
		\int_{0}^{\infty } \int_{\mathbb{R} }^{} \rho ^{\varepsilon} \varphi _{t} +\rho ^{\varepsilon} u ^{\varepsilon} \varphi _{x} dxdt+\int_{\mathbb{R} }^{} \rho _{0} \varphi \left ( x,0 \right ) dx=0,
	\end{equation}
	\begin{equation}
		\int_{0}^{\infty } \int_{\mathbb{R} }^{} \rho ^{\varepsilon}\left ( u^{\varepsilon}+\varepsilon ^{2} p\left(\rho^{\varepsilon}\right) \right ) \varphi _{t} +\rho ^{\varepsilon} u ^{\varepsilon} \left ( u^{\varepsilon}+\varepsilon ^{2} p\left(\rho^{\varepsilon}\right) \right ) \varphi _{x} dxdt+\int_{\mathbb{R} }^{} \rho _{0} \left ( u_{0}+\varepsilon ^{2} p\left(\rho_{0}\right) \right ) \varphi \left ( x,0 \right ) dx=0.
	\end{equation}
	As $\varepsilon\to 0$, by the convergence of $\left(\rho^{\varepsilon},u^{\varepsilon}\right)$, the above equalities turn to
	\begin{equation}
		\int_{0}^{\infty } \int_{\mathbb{R} }^{} \bar{\rho} \varphi _{t} + \bar{\rho} \bar{u} \varphi _{x} dxdt+\int_{\mathbb{R} }^{} \rho _{0} \varphi \left ( x,0 \right ) dx=0,
	\end{equation}
	\begin{equation}
		\int_{0}^{\infty } \int_{\mathbb{R} }^{} \bar{\rho} \bar{u} \varphi _{t} + \bar{\rho} \bar{u}^{2}\varphi _{x} dxdt+\int_{\mathbb{R} }^{} \rho _{0} u_{0} \varphi \left ( x,0 \right ) dx=0,
	\end{equation}
	which show $\left(\bar{\rho},\bar{u}\right)$ satisfies \eqref{pressureless Euler} with initial value \eqref{intial value of AR} in $H$.
	Finally, we consider the convergence of $u_{x}^{\varepsilon}$. For $\varphi\in C_{c}^{\infty}\left(H\right)$, as $\varepsilon\to 0$, 
	\begin{equation}
		\int_{0}^{\infty } \int_{\mathbb{R} }^{} \left ( u_{x}^{\varepsilon }-\bar{u} _{x} \right ) \varphi dxdt = -\int_{0}^{\infty } \int_{\mathbb{R} }^{} \left ( u^{\varepsilon }-\bar{u}\right ) \varphi_{x}  dxdt \to 0. 
	\end{equation}
	So we have $u_{x}^{\varepsilon } \rightharpoonup \bar{u} _{x} $. Since for each $H$, $u_{x}^{\varepsilon }$ are uniformly bounded respect to $\varepsilon$, then $u_{x}^{\varepsilon } \rightarrow \bar{u} _{x} $ in $L^\infty(H)$.
	
	Then, in Lagrangian coordinates, we could have as $\varepsilon \to 0$
	\begin{equation}
		m^{\varepsilon}\left(\tau\right) \to 
		m\left(\tau\right) := \inf\limits_{y\in\mathbb{R}} \left \{ \bar{J}^{-1}\bar{v}_{y} \left ( y,\tau \right ) \right \} \,\, and \,\, \lim_{\tau\uparrow T_{b} } m \left ( \tau  \right ) =-\infty ,
	\end{equation}
	where $T_{b}$ is defined in \eqref{Tb}.
	Similar to $\tau_{M}^{\varepsilon}$, for $M>0$, we could introduce
	\begin{equation}
		\tau_{M}=\sup\left\{s:-M\le \inf\limits_{\tau\in\left[0,s\right]}m\left(\tau\right),s\le T\right\}.
	\end{equation}
	According to the definition, $\left\{\tau_{M}\right\}$ is a monotone sequence with respect to $M$, and as $M\to\infty$
	\begin{equation}\label{convergence of tM}
		\lim_{M \to +\infty } \tau _{M} =\min_{ } \left \{ T_{b},T\right \}=: \tau _{b}.
	\end{equation}
	By the convergence of the level set, we have 
	\begin{equation}\label{convergence of tMe}
		\lim_{\varepsilon \to 0} \tau _{M}^{\varepsilon} = \tau_{M}.
	\end{equation}
	
	Next, we shows that 
	\begin{equation}\label{relationship of Tb}
		T_{b} \le \varliminf_{\varepsilon\to0} T_{b}^{\varepsilon}.
	\end{equation}
	Since $\tau_{M}^{\varepsilon} \le T_{b}^{\varepsilon}.$ So we have
	\begin{equation}\label{convergence of tM}
		\lim_{M \to +\infty } \lim_{\varepsilon \to 0 } \tau _{M}^{\varepsilon} \le \lim_{M \to +\infty }\varliminf_{\varepsilon\to0} T_{b}^{\varepsilon}.
	\end{equation}
	The left hand side holds 
	\begin{equation}
		\lim_{M \to +\infty } \lim_{\varepsilon \to 0 } \tau _{M}^{\varepsilon} =
		\lim_{M \to +\infty } \tau_{M} =T_{b}.
	\end{equation}
	So we get \eqref{relationship of Tb}.

	\section{Convergence of blow-up time}
	In this section, if $\rho_{0}\in C^1_{s} \left(\mathbb{R}\right)$, $u_{0}\in C^{1} \left(\mathbb{R}\right)$, we can further obtain the convergence of the blow-up time $T_{b}^{\varepsilon}$. First, we need to introduce the non-discontinuous regions as:
	$$\Omega_{M}^{\varepsilon} = \Omega_{+} \cup\Omega^{\varepsilon}_{\amalg }.$$
	Similar to the previous definition, we define 
	\begin{equation}
		\bar{m}^{\varepsilon}\left(\tau\right):=\inf\limits_{(y, \tau)\in \Omega_{M}^{\varepsilon}} \left \{ \left(J^{\varepsilon}\right)^{-1}v_{y}^{\varepsilon} \left ( y,\tau  \right ) \right \} .
	\end{equation}
	For any fixed $\varepsilon>0$ and the compressive initial data, there exists a finite life-span $\bar{T}_{b}^{\varepsilon}$, which is $+\infty$ for the rarefaction initial data.
	And we have
	\begin{equation}
		\lim_{\tau\uparrow \bar{T}_{b}^{\varepsilon } } \bar{m}^{\varepsilon} \left ( \tau  \right ) =-\infty .
	\end{equation}
	Further, for $M>0$, we define $\bar{\tau}_{M}^{\varepsilon}$
	\begin{equation}
		\bar{\tau}_{M}^{\varepsilon}=\sup\left\{s:-M\le \inf\limits_{\tau\in\left[0,s\right]}\bar{m}^{\varepsilon}\left(\tau\right),s\le T\right\}.
	\end{equation}
	According to the definition, $\left\{\bar{\tau}_{M}^{\varepsilon}\right\}$ is a monotone sequence with respect to $M$, and
	\begin{equation}
		\lim_{M \to +\infty } \bar{\tau} _{M}^{\varepsilon} =\min_{ } \left \{ \bar{T}_{b}^{\varepsilon },T\right \}=:\bar{\tau} _{b}^{\varepsilon } .
	\end{equation}
	Before discussing the convergence of the blow-up time, we need to discuss the modulus of continuity estimation in $\Omega_{M}^{\varepsilon}$.
	
	The modulus of continuity of a function $f\left(x,t\right)$ is defined by the following non-negative function:
	\begin{equation}
		\eta \left(f\mid \varepsilon\right)\left(s\right) = 
		\sup \left\{\left | f\left(x_{1},s\right) - f\left(x_{2},s\right) \right | : s\in \left[0,t\right], \left | x_{1} - x_{2} \right |<\varepsilon \right\}.
	\end{equation}
	For modulus of continuity, there are the following properties, for more details please refer to \cite{Li Ta-Tsien Yu Wenci}:
	\begin{lemma}\label{modulus of continuity}
		For modulus of continuity
		
		(1) If $\epsilon_{1}<\epsilon_{2}$, $\eta\left(f\mid \epsilon_{1}\right) < \eta\left(f\mid \epsilon_{2}\right)$.
		
		(2) For any positive number $C$, $\eta\left(f\mid C\epsilon \right) \le \left(C+1\right)\eta\left(f\mid \epsilon \right)$.
		
		(3) For $f$ and $g$ that are two continuous functions, 
		$\eta\left(f\pm g\mid \epsilon\right) \le \eta\left(f\mid \epsilon\right) + \eta\left(g\mid \epsilon\right)$, and
		$\eta\left(fg\mid\epsilon\right)\le \eta\left(f\mid \epsilon\right)\left \| g \right \|_{L^{\infty}} +\eta \left(g\mid \epsilon\right)\left \| f \right \|_{L^{\infty}}$.
		
		(4) If $f\in C^{\alpha}$, $0<\alpha\le 1$, if and only if there exists a constant $L$ such that $\eta\left(f\mid \epsilon\right)\le L\epsilon^{\alpha}$.
	\end{lemma}
	Define 
	\begin{equation}
		\eta \left(f_{1},f_{2},\dots ,f_{n}\mid\epsilon\right)\left(t\right):=
		\max \left\{\eta\left(f_{i}\mid\epsilon\right)\left(t\right): i=1, \dots, n\right\}.
	\end{equation}
	By definition, $\eta \left(f_{1},f_{2},\dots ,f_{n}\mid\epsilon\right)\left(t\right)$ has similar properties in Lemma \ref{modulus of continuity}. So we have the following lemma:
	\begin{lemma}
		If $u_0\in C^1(\mathbb{R})$ and $g_0\in C^1(\mathbb{R_-\cup\mathbb{R}_+})$. Then on the $\Omega_{M}^{\varepsilon}$, the modulus of continuity of $\left(J^{\varepsilon}\right)^{-1}v_{y}^{\varepsilon}$ and $v^{\varepsilon}_{\tau}$ is uniform, which means $\left(J^{\varepsilon}\right)^{-1}v_{y}^{\varepsilon}$ and $v^{\varepsilon}_{\tau}$ is equi-continuous.
	\end{lemma}
	\begin{proof}

		By \eqref{AR}, we have the following equation
		\begin{equation}\label{modulus}
			\begin{split}
				D\left(\frac{v^\ep_y}{J^\ep}\right)=-\left(2+\frac{g^{\varepsilon}p''(g^{\varepsilon})}{p'(g^{\varepsilon})}\right)\left(\frac{v^\ep_y}{J^\ep}\right)^2+\left(1+\frac{g^{\varepsilon}p''(g^{\varepsilon})}{p'(g^{\varepsilon})}\right)\frac{g^{\varepsilon}}{g_0}Z^{\varepsilon}_y\frac{v^\ep_y}{J^\ep}.
			\end{split}
		\end{equation}
		And $v_\tau$ can be expressed by $\frac{v^\ep_y}{J^\ep}$ and the lower order terms. Thus we only need to prove the modulus of continuity of $\frac{v^\ep_y}{J^\ep}$. To prove this, we first prove the modulus of continuity of the characteristics line. Then combine with the regularity of the lower order terms, we obtain the modulus of continuity of $\frac{v^\ep_y}{J^\ep}$.

		By the lemma \ref{uniform estimations 2}, we have that $g^\varepsilon$ and $J^\varepsilon$ are uniformly Lipschitzs functions on the $\Omega_M^\ep$. Thus for the characteristic line, for $\epsilon>0$ and $\left | \tilde{y}_{1}-\tilde{y}_{2} \right | < \epsilon $, without loss of generality, we assume $s<\tau$
		\begin{eqnarray}
			& & \left | y_{1}^{\varepsilon}\left(\tau;\tilde{y}_{1},s\right) -y_{1}^{\varepsilon}\left(\tau;\tilde{y}_{2},s\right)\right | \nonumber \\
			& \le & \left | \tilde{y}_{1}-\tilde{y}_{2}\right |
			+ \int_{s}^{\tau} \left | \mu^{\varepsilon}\left(y_{1}^{\varepsilon}\left(s';\tilde{y}_{1},s\right),s'\right) -
			\mu^{\varepsilon}\left(y_{1}^{\varepsilon}\left(s';\tilde{y}_{2},s\right),s'\right) \right | ds' \nonumber \\
			& \le & \left | \tilde{y}_{1}-\tilde{y}_{2}\right |
			+ C\left(M\right) \int_{s}^{\tau} 
			\left | y_{1}^{\varepsilon}\left(s';\tilde{y}_{1},s\right) -
			y_{1}^{\varepsilon}\left(s';\tilde{y}_{2},s\right) \right | ds'.
		\end{eqnarray}
		By $\rm{Gr\ddot{o}nwall}$'s inequality, we conclude that
		\begin{equation}\label{cha-mod}
			\left | y_{1}^{\varepsilon}\left(\tau;\tilde{y}_{1},s\right) -y_{1}^{\varepsilon}\left(\tau;\tilde{y}_{2},s\right)\right | 
			\le \left | \tilde{y}_{1}-\tilde{y}_{2}\right | e^{C\left(M\right)\left(\tau-s\right)} \le 
			\left | \tilde{y}_{1}-\tilde{y}_{2}\right | e^{C\left(M\right)\tau}.
		\end{equation}
		Then we consider the modulus of continuity of $\frac{v^\ep_y}{J^\ep}$. We consider the estimation on two directions: characteristic direction and $y$-direction. For characteristic direction, because $g^\ep$ and $v^\ep$ are uniform bounded in $C^1(\Omega_M^\ep)$,
		
		\begin{equation}
			\left | \frac{v^\ep_y}{J^\ep}\left(y_{1}^{\varepsilon}\left(\tau_{1};\tilde{y},\tilde{\tau}\right),\tau_{1}\right) -
			\frac{v^\ep_y}{J^\ep}\left(y_{1}^{\varepsilon}\left(\tau_{2};\tilde{y},\tilde{\tau}\right),\tau_{2}\right)\right | \le 
			\left | \int_{\tau_{1}}^{\tau_{2}}D\left(\frac{v^\ep_y}{J^\ep}\right)ds\right | 
			\le C\left(M\right)\left | \tau_{1}-\tau_{2} \right | .
		\end{equation}
		Now we consider the $y-$direction in which we need the \eqref{cha-mod}. For convenient, we use the denotation
		$Y^i\left(s\right)=\left(y^i\left(s\right),s\right)=\left(y^\ep_1\left(s;\tilde{y}_i,0\right),s\right)$. Thus for any $\left(\tilde{y}_i,0\right)\in\Omega_M^\varepsilon$, $i=1,2$, Let 
		\begin{equation}
			G\left(g^{\varepsilon}\right) = \frac{g^{\varepsilon}p''(g^{\varepsilon})}{p'(g^{\varepsilon})},
		\end{equation}
		By \eqref{modulus} we have
		\begin{eqnarray}
			& & \left|\frac{v^\ep_y}{J^\ep}\left(Y^1\left(\tau\right)\right)-\frac{v^\ep_y}{J^\ep}\left(Y^2\left(\tau\right)\right)\right| \nonumber \\
			& \le & \left|v'_{0}\left(\tilde{y_{1}}\right)-v'_{0}\left(\tilde{y_{2}}\right)\right| +
			C(M)\int_0^{\tau} \left|\left(Z_{0}^{\varepsilon}\right)_{y}\left(y^1\left(s\right)\right)-\left(Z_{0}^{\varepsilon}\right)_{y}\left(y^2\left(s\right)\right)\right|ds \nonumber \\
			& & + C\left(M\right)\int_0^{\tau} \left(\left|g^{\varepsilon}\left(Y^1\left(s\right)\right)-g^{\varepsilon}\left(Y^2\left(s\right)\right)\right| + \left|G\left(g^{\varepsilon}\left(Y^1\left(s\right)\right)\right)-G\left(g^{\varepsilon}\left(Y^2\left(s\right)\right)\right)\right|\right) ds \nonumber \\
			& & +  C\left(M\right)\int_0^{\tau}\left|\frac{v^\ep_y}{J^\ep}\left(Y^1\left(s\right)\right)-\frac{v^\ep_y}{J^\ep}\left(Y^2\left(s\right)\right)\right|ds \nonumber\\
			& \le & \eta\left(v'_{0}\big|\left|\tilde{y_{1}}-\tilde{y_{2}}\right|\right)\left(0\right)
			+ C(M) \int_0^{\tau} \eta\left(\left(Z_{0}^{\varepsilon}\right)_{y}\Big|\left|y^1\left(s\right)-y^2\left(s\right)\right|\right)\left(0\right) ds \nonumber \\
			&& + C(M) \int_0^{\tau} \left(\eta\left(g^{\varepsilon}\Big|\left|Y^1\left(s\right)-Y^2\left(s\right)\right|\right)\left(s\right) + \eta\left(G\left(g^{\varepsilon}\right)\Big|\left|Y^1\left(s\right)-Y^2\left(s\right)\right|\right)\left(s\right)\right) ds \nonumber \\
			&& + C(M) \int_0^{\tau} \eta\left(\frac{v^\ep_y}{J^\ep}\Big|\left|Y^1\left(s\right)-Y^2\left(s\right)\right|\right)\left(s\right) ds.
		\end{eqnarray}
		For first term we have
		\begin{equation}
			\eta\left(v'_{0}\big|\left|\tilde{y_{1}}-\tilde{y_{2}}\right|\right)\left(0\right) \le \eta\left(v'_{0}|\epsilon\right)\left(0\right).
		\end{equation}
		The second term we have 
		\begin{eqnarray}
			\eta\left(\left(Z_{0}^{\varepsilon}\right)_{y}\Big|\left|y^1\left(s\right)-y^2\left(s\right)\right|\right)\left(0\right) \le 
			\eta\left(\left(Z_{0}^{\varepsilon}\right)_{y}\Big|e^{C(M)s}\epsilon\right)\left(0\right) \le
			\left(e^{C(M)s}+1\right)\eta\left(\left(Z_{0}^{\varepsilon}\right)_{y}\Big|\epsilon\right)\left(0\right).
		\end{eqnarray}
		For the third term we have
		\begin{eqnarray}
			\eta\left(g^{\varepsilon}\Big|\left|Y^1\left(s\right)-Y^2\left(s\right)\right|\right)\left(s\right) \le C(M) \left|y^1\left(s\right)-y^2\left(s\right)\right| \le C(M)e^{C(M)s}\epsilon.
		\end{eqnarray}
		and
		\begin{eqnarray}
			\eta\left(G\left(g^{\varepsilon}\right)\Big|\left|Y^1\left(s\right)-Y^2\left(s\right)\right|\right)\left(s\right) \le \eta\left(G\left(g^{\varepsilon}\right)\Big|C(M)e^{C(M)s}\epsilon\right)\left(s\right) \le
			\left(C(M)e^{C(M)s}+1\right)\eta\left(G\Big|\epsilon\right)\left(s\right).
		\end{eqnarray}
		For the last term, we have
		\begin{equation}
			\begin{split}
				\eta\left(\frac{v^\ep_y}{J^\ep}\Big|\left|Y^1\left(s\right)-Y^2\left(s\right)\right|\right)\left(s\right) & \le
				\eta\left(\frac{v^\ep_y}{J^\ep}\Big|e^{C(M)s}\epsilon\right)\left(s\right) \le
				\left(e^{C(M)s}+1\right)\eta\left(\frac{v^\ep_y}{J^\ep}\Big|\epsilon\right)\left(s\right).
			\end{split}
		\end{equation}
		There exists a constant $C'(M)$ such that
		\begin{eqnarray}
			\eta\left(\frac{v^\ep_y}{J^\ep}\Big|\epsilon\right)\left(s\right) \le 
			C'(M) \left(\left(\epsilon+\eta\left(v'_{0},\left(Z_{0}^{\varepsilon}\right)_{y}\Big|\epsilon\right)\left(0\right)\right) + \int_{0}^{\tau} \eta\left(G\Big|\epsilon\right)\left(s\right) ds \right) + C'(M) \int_{0}^{\tau} \eta\left(\frac{v^\ep_y}{J^\ep}\Big|\epsilon\right)\left(s\right) ds.
		\end{eqnarray}
		
		For any $\varepsilon>0$, because $\tilde{y}_i$ are chosen arbitrary on the $\Omega_M^\ep$. Through the $\rm{Gr\ddot{o}nwall}$'s inequality, the above inequality equals to
		\begin{equation}
			\eta\left(\frac{v^\ep_y}{J^\ep}\bigg|\epsilon\right)\le C'(M)e^{C'(M)T}
			\left(\epsilon + \eta\left(v'_{0},\left(Z_{0}^{\varepsilon}\right)_{y}\Big|\epsilon\right)\left(0\right) + \int_{0}^{\tau} \eta\left(G\Big|\epsilon\right)\left(s\right) ds\right).
		\end{equation}
		Since $G$ is a continuous function, the modulus of continuity of $\frac{v^\ep_y}{J^\ep}$ on $\Omega_{M}^{\varepsilon}$ is uniform.

	\end{proof}

	For a compact set $H\subset \Omega_{M}^{\varepsilon}$, we could lift the convergence as $\left(J^{\varepsilon}\right)^{-1}v_{y}^{\varepsilon} \to \bar{J}^{-1}\bar{v}_{y}$ in $C\left(H\right)$ as $\varepsilon \to 0$. Then, for $(y,\tau)\in H$, we could have as $\varepsilon \to 0$, 
	\begin{equation}
		\bar{m}^{\varepsilon}\left(\tau\right) \to 
		m\left(\tau\right).
	\end{equation}
	By the convergence of the level set, we have 
	\begin{equation}\label{convergence of tMe}
		\lim_{\varepsilon \to 0} \bar{\tau} _{M}^{\varepsilon} = \tau_{M}.
	\end{equation}
	Furthermore, we consider the convergence of the blow-up time for $I(g)$ satisfying the following conditions:
	\begin{enumerate}[label=(\alph*)]
		\item $I\left(g\right)$ is a  increasing function with respect to $g$;
		\item There exists a small $\delta$ such that $\int_{0}^{\delta}\frac{I(s)}{s^{2}}ds=+\infty$.
	\end{enumerate}

	For example, $p(g)=\ln{g} $ meets the above conditions. Using the idea in \cite{Peng Wenjian Wang Tian-Yi Xiang Wei}, we have the following lemma.
	\begin{lemma}
		For $M>0$, if $I(g)$ meet the above conditions (1) and (2), then we have
		\begin{equation}
			\lim_{\varepsilon \to 0} T _{b}^{\varepsilon} = T_{b}.
		\end{equation}
	\end{lemma}
	\begin{proof}
		Case 1: $\inf\left\{{v_{0}'\left(y\right)}\right\}\ge 0$.
		
		In this case, $T_{b}=+\infty$. Let $\alpha^{\varepsilon}:= g^{\varepsilon}p'\left(g^{\varepsilon}\right)\left(J^{\varepsilon}\right)^{-1}v_{y}^{\varepsilon}$, 
		by \eqref{blow-up of general pressure}
		\begin{equation}\label{alpha>0}
			\alpha^{\varepsilon}\left(y,0\right) \\
			= \alpha_{0}\left(y\right) = g_{0}p'\left(g_{0}\right) v_{0}'\left(y\right) \ge 0.
		\end{equation}
		In this case we want to show $\lim_{\varepsilon  \to 0} T_{b}^{\varepsilon} = + \infty$. 
		If the claim is false, we can find a subsequence $\left\{\varepsilon_{n}\right\}$ such that $\left\{T_{b}^{\varepsilon_{n}}\right\}$ is bounded. For any $\varepsilon_{n}$, there exist $y_{1}^{\varepsilon_{n}}$ defined by \eqref{characteristic line} such that when $\tau \uparrow T_{b}^{\varepsilon_{n}}$,
		\begin{equation}
			\lim_{\tau \uparrow T_{b}^{\varepsilon_{n}}} \alpha^{\varepsilon_{n}} = -\infty,
		\end{equation}
		which
		\begin{equation}
			1+\alpha^{\varepsilon_{n}}\left(y_{1}^{\varepsilon_{n}}\left(0;y,\tau\right),0\right)\int_{0}^{T_{b}^{\varepsilon_{n}}}I\left(g^{\varepsilon_{n}}\left(y_{1}^{\varepsilon_{n}}\left(s;y,\tau\right),s\right)\right)ds = 0.
		\end{equation}
		The above equation is equivalent to
		\begin{equation}
			- \frac{1}{\alpha^{\varepsilon_{n}}\left(y_{1}^{\varepsilon_{n}}\left(0;y,\tau\right)\right)} = \int_{0}^{T_{b}^{\varepsilon_{n}}}I\left(g^{\varepsilon_{n}}\left(y_{1}^{\varepsilon_{n}}\left(s;y,\tau\right),s\right)\right)ds.
		\end{equation}
		The right hand side is uniformly bounded by a finite positive constant
		\begin{equation}
			\int_{0}^{T_{b}^{\varepsilon_{n}}}I\left(g^{\varepsilon_{n}}\left(y_{1}^{\varepsilon_{n}}\left(s;y,\tau\right),s\right)\right)ds \le C.
		\end{equation}
		For $\alpha^{\varepsilon_{n}}$
		\begin{equation}
			\alpha^{\varepsilon_{n}}\left(y_{1}^{\varepsilon_{n}}\left(0;y,\tau\right)\right) \le -\frac{1}{C} < 0,
		\end{equation}
		This contradicts with \eqref{alpha>0}, so we have
		
		\begin{equation}
			\lim_{\varepsilon  \to 0} T_{b}^{\varepsilon} = + \infty = T_{b}.
		\end{equation}
		
		Case 2: $\inf\left\{{v_{0}'\left(y\right)}\right\}< 0$.
		
		For large $M$, we first proof that $\bar{T}_{b}^{\varepsilon}$ is uniformly bounded.
		If the claim is false, we can find a subsequence $\left\{\varepsilon_{n}\right\}$ such that $\bar{T}_{b}^{\varepsilon_{n}} \to + \infty$ as $n\to +\infty$.
		And there exist $y_{1}^{\varepsilon_{n}}$ defined by \eqref{characteristic line} such that when $\tau \uparrow \bar{T}_{b}^{\varepsilon_{n}}$, 
		\begin{equation}
			1+\alpha^{\varepsilon_{n}}\left(y_{1}^{\varepsilon_{n}}\left(0;y,\tau\right),0\right)\int_{0}^{\bar{T}_{b}^{\varepsilon_{n}}}I\left(g^{\varepsilon_{n}}\left(y_{1}^{\varepsilon_{n}}\left(s;y,\tau\right),s\right)\right)ds = 0.
		\end{equation}
		By \eqref{The lower bound independent of epsilon} and the monotonicity of $I(\rho)$, there exists a constant $C$ such that
		\begin{equation}
			C \ge -\frac{1}{\alpha^{\varepsilon_{n}}\left(y_{1}^{\varepsilon_{n}}\left(0;y,\tau\right)\right)} = \int_{0}^{\bar{T}_{b}^{\varepsilon_{n}}}I\left(g^{\varepsilon_{n}}\left(y_{1}^{\varepsilon_{n}}\left(s;y,\tau\right),s\right)\right)ds
			\ge \int_{0}^{\bar{T}_{b}^{\varepsilon_{n}}}I\left(\frac{A_{1}}{1+A_{2}Bs}\right)ds .
		\end{equation}
		Let $n\to \infty$ we have
		\begin{equation}
			C \ge \int_{0}^{+\infty}I\left(\frac{A_{1}}{1+A_{2}Bs}\right) ds 
			= A_{2}B\int_{0}^{+\infty}I\left(\frac{A_{1}}{1+s}\right)ds.
		\end{equation}
		The definitions of $A_{1}$, $A_{2}$ and $B$ see proposition \ref{vanishing pressure limit lower bound}.
		Let $y=\frac{A_{1}}{1+s}$ we have
		\begin{equation}\label{bounded contradiction of integral}
			C \ge A_{1}A_{2}B\int_{0}^{A_{1}}\frac{I\left(y\right)}{y^{2}}ds 
			= A_{2}B\int_{0}^{+\infty}I\left(\frac{A_{1}}{1+s}\right)ds.
		\end{equation}
		On the other hand, there exists $\delta$ such that
		\begin{equation}
			\int_{0}^{\delta}\frac{I\left(y\right)}{y^{2}}ds 
			= +\infty.
		\end{equation}
		This contradicts with \eqref{bounded contradiction of integral}, so $\bar{T}_{b}^{\varepsilon}$ is bounded with respect to $\varepsilon$.

		We set
		\begin{equation}
			-2\delta_{0}:=\inf\left\{v_{0}'\right\}<0.
		\end{equation}
		where $\delta_{0}$ is a constant independent from $\varepsilon$. 
		When $\varepsilon$ is small
		enough such that
		\begin{equation}
			\varepsilon^{2}\left \| \left(p\left(g_{0}\right)\right)_{y} \right \|_{L^{\infty}} < \frac{\delta_{0}}{2}.
		\end{equation}
		Then on $\Omega_{M}^{\varepsilon}$ there exists $y_{1}^{\varepsilon}\left(s;\tilde{y},0\right)$ 
		\begin{equation}
			v_{0}'\left(\tilde{y}\right)=-2\delta_{0},\quad
		\end{equation}
		when $\varepsilon<\varepsilon_{0}$ is small enough, $\tilde{y}$ and $y^{\varepsilon}_{1}\left(\bar{T}_{b}^{\varepsilon};\tilde{y},0\right)$ are close enough such that for $y\in \left[y^{\varepsilon}_{1}\left(\bar{T}_{b}^{\varepsilon};\tilde{y},0\right), \tilde{y}\right]$
		\begin{equation}
			v_{0}'\left(y\right) \le -\frac{3}{2}\delta_{0}.
		\end{equation}
		For $y\in \left[y^{\varepsilon}_{1}\left(\bar{T}_{b}^{\varepsilon};\tilde{y},0\right), \tilde{y}\right]$ we have
		\begin{equation}
			\left(Z^{\varepsilon}_{0}\right)_{y}\left(y\right) = v_{0}'\left(y\right) + \varepsilon^{2}\left(p\left(g_{0}\left(y\right)\right)\right)_{y} \le - \delta_{0}.
		\end{equation}

		Next, we want to show, for $\varepsilon$ small enough, there exists $M_{0}$ independent of $M$ such that
		\begin{equation}
			\bar{T}_{b}^{\varepsilon} - \bar{\tau}_{M}^{\varepsilon} \le \frac{M_{0}}{M}.
		\end{equation}
		For $g^{\varepsilon}$ we have
		\begin{equation}
			\int_{\bar{\tau}^{\varepsilon}_{M}}^{\bar{T}_{b}^{\varepsilon}} D\left(\frac{1}{g^{\varepsilon}}\right)\left(y^{\varepsilon}_{1}\left(s;\tilde{y},0\right),s\right) ds = \int_{\bar{\tau}^{\varepsilon}_{M}}^{\bar{T}_{b}^{\varepsilon}} g_{0}^{-1}\left(Z^{\varepsilon}_{0}\right)_{y} \left(y^{\varepsilon}_{1}\left(s;\tilde{y},0\right),s\right) ds.
		\end{equation}
		The above equation is equivalent to
		\begin{eqnarray}\label{upper bound of integral}
			\frac{1}{g^{\varepsilon}\left(y_{1}^{\varepsilon}\left(\bar{\tau}_{M}^{\varepsilon};\tilde{y},0\right),\bar{\tau}_{M}^{\varepsilon}\right)} & \ge & \frac{1}{g^{\varepsilon}\left(y_{1}^{\varepsilon}\left(\bar{\tau}_{M}^{\varepsilon};\tilde{y},0\right),\bar{\tau}_{M}^{\varepsilon}\right)} - \frac{1}{g^{\varepsilon}\left(y_{1}^{\varepsilon}\left(\bar{T}_{b}^{\varepsilon};\tilde{y},0\right),\bar{T}_{b}^{\varepsilon}\right)} \nonumber \\
			& = & \int_{\bar{\tau}_{M}^{\varepsilon}}^{\bar{T}_{b}^{\varepsilon}} - g_{0}^{-1}\left(Z_{0}^{\varepsilon}\right)_{y}\left(y_{1}^{\varepsilon}\left(s;\tilde{y},0\right)\right) \nonumber \\
			& \ge & \frac{\delta_{0}}{\bar{g}_{0}}\left(\bar{T}_{b}^{\varepsilon} - \bar{\tau}_{M}^{\varepsilon} \right).
		\end{eqnarray}
		When $M$ is large enough, by the convergence of $\left(g^{\varepsilon},v^{\varepsilon}\right)$ and the explicit expression of $\left(\bar{g},\bar{v}\right)$, we have for $\tau\in\left[0,\tau_{M}\right]$
		\begin{equation}
			\left | \frac{\bar{J}^{-1}\bar{v}_{y}\left(\tilde{y},\tau_{M}\right)}
			{g^{\varepsilon}\left(y_{1}^{\varepsilon}\left(\bar{\tau}_{M}^{\varepsilon};\tilde{y},0\right),\bar{\tau}_{M}^{\varepsilon}\right)} \right | \to 
			\left | \frac{v'_{0}\left(y\right)}
			{g_{0}\left(y\right)} \right | .
		\end{equation}
		There exists $M_{0}$ independent of $M$ such that
		\begin{equation} 
			\left | \frac{v'_{0}\left(y\right)}
			{g_{0}\left(y\right)} \right | < \frac{M_{0}}{2}.
		\end{equation}
		Then, for $\varepsilon$ small enough 
		\begin{equation}
			\left | \frac{\bar{J}^{-1}\bar{v}_{y}\left(\tilde{y},\tau_{M}\right)}
			{g^{\varepsilon}\left(y_{1}^{\varepsilon}\left(\bar{\tau}_{M}^{\varepsilon},\tilde{y},0\right),\bar{\tau}_{M}^{\varepsilon}\right)} \right | < M_{0}.
		\end{equation}
		Thus
		\begin{equation}\label{Tb-tau_M}
			\frac{1}
			{g^{\varepsilon}\left(y_{1}^{\varepsilon}\left(\bar{\tau}_{M}^{\varepsilon},\tilde{y},0\right),\bar{\tau}_{M}^{\varepsilon}\right)} < \frac{M_{0}}{\left | \bar{J}^{-1}\bar{v}_{y}\left(\tilde{y},\tau_{M}\right) \right |}
			= \frac{M_{0}}{M}.
		\end{equation}
		By \eqref{upper bound of integral} and \eqref{Tb-tau_M} we obtain the following inequality
		\begin{equation}
			\bar{T}_{b}^{\varepsilon} - \bar{\tau}_{M}^{\varepsilon} \le \frac{\bar{g}_{0}}{\delta_{0}}\cdot\frac{M_{0}}{M}.
		\end{equation}
		
		Next, we prove that $\lim_{\varepsilon \to 0} \bar{T} _{b}^{\varepsilon} = T_{b}$.
		There exists $y$ such that
		\begin{equation}
			\tau_{M} = -\frac{1}{v_{0}'\left(y\right)}-\frac{1}{M},\quad
			T_{b} = -\frac{1}{v_{0}'\left(y\right)}.
		\end{equation}
		Thus we have
		\begin{equation}
			\left | \tau_{M} - T_{b} \right | \le \frac{1}{M}.
		\end{equation}
		Then we can choose $M$ large enough such that
		\begin{equation}
			\left | \bar{T}_{b}^{\varepsilon} - \bar{\tau}_{M}^{\varepsilon} \right |
			+ \left | \tau_{M} - T_{b} \right | < \frac{\delta}{2}.
		\end{equation}
		Since $\left(g^{\varepsilon},v^{\varepsilon}\right)$ uniformly converges to $\left(\bar{g},\bar{v}\right)$ on
		any level set, there exists $\varepsilon_{1}$ such that for $\varepsilon\in\left(0,\varepsilon_{1}\right]$, we have
		\begin{equation}
			\left | \bar{\tau}_{M}^{\varepsilon} - \tau_{M} \right | < \frac{\delta}{2}.
		\end{equation}
		Combined with the above discussion, when $M$ is sufficiently large, for any $\delta>0$, there exists $\varepsilon_{1}$ such that for $\varepsilon\in\left(0,\varepsilon_{1}\right]$
		\begin{equation}
			\left | \bar{T}_{b}^{\varepsilon} - T_{b} \right | \le 
			\left | \bar{T}_{b}^{\varepsilon} - \bar{\tau}_{M}^{\varepsilon} \right |
			+ \left | \bar{\tau}_{M}^{\varepsilon} - \tau_{M} \right |
			+ \left | \tau_{M} - T_{b} \right | < \delta.
		\end{equation}

		Therefore, we have
		\begin{equation}
			\lim_{\varepsilon \to 0} \bar{T} _{b}^{\varepsilon} = T_{b}.
		\end{equation}
		Finally, we need to prove $\lim_{\varepsilon \to 0} T _{b}^{\varepsilon} = T_{b}$.
		By the definition of $\bar{T}_{b}^{\varepsilon}$, we have
		\begin{equation}
			T_{b} = \lim_{\varepsilon\to0} \bar{T}_{b}^{\varepsilon} \ge
			\varlimsup_{\varepsilon\to0} T_{b}^{\varepsilon} \ge
			\varliminf_{\varepsilon\to0} T_{b}^{\varepsilon}.
		\end{equation}
		On the other hand, by \eqref{relationship of Tb} we have
		\begin{equation}
			T_{b} \le \varliminf_{\varepsilon\to0} T_{b}^{\varepsilon}.
		\end{equation}
		Therefore, we have
		\begin{equation}
			\varlimsup_{\varepsilon\to0} T_{b}^{\varepsilon} =
			\varliminf_{\varepsilon\to0} T_{b}^{\varepsilon} = \lim_{\varepsilon \to 0} T _{b}^{\varepsilon} = T_{b}.
		\end{equation}
		
	\end{proof}

	\section{Convergence Rates}
	In this section, we will consider the convergence rates of solution on each $\left[0,t_{M}\right]$ in Eulerian coordinates.
	For $u^{\varepsilon}$ and $z^{\varepsilon}$, we have the following lemma.
	\begin{lemma}
		For $M>0$, on $\mathbb{R}\times\left[0,t_{M}\right]$, $u^{\varepsilon } \sim z^{\varepsilon } \sim \lambda^{\varepsilon}_{1} \sim \lambda^{\varepsilon}_{2} \sim \bar{x} \sim x_{2}^{\varepsilon} \sim \bar{u} \left ( O\left ( \varepsilon ^{2}  \right )  \right ) $.
	\end{lemma}
	\begin{remark}
		$\lambda^{\varepsilon}_{1} \sim \lambda^{\varepsilon}_{2} \sim \bar{u}\left(O\left(\varepsilon^{2}\right)\right)$, the characteristics triangle vanishing as the $\varepsilon^{2}$ order.
	\end{remark}
	\begin{proof}
		In order to simplify the calculation, we set
		\begin{equation} 
			\begin{cases}
				x_{1}^{\varepsilon}\left(s\right) = x_{1}^{\varepsilon}\left(s;\tilde{x},\tilde{t}\right), \\
				x_{2}^{\varepsilon}\left(s\right) = x_{2}^{\varepsilon}\left(s;\tilde{x},\tilde{t}\right), \\
				\bar{x}^{\varepsilon}\left(s\right) = \bar{x}^{\varepsilon}\left(s;\tilde{x},\tilde{t}\right), \\
			\end{cases}
		\end{equation}
		which are defined at \eqref{x-i} and \eqref{bar x}.
		On $\mathbb{R}\times\left[0,t_{M}\right]$
		\begin{eqnarray}
			\left | x_{1}^{\varepsilon}\left(s\right) - x_{2}^{\varepsilon}\left(s\right) \right | 
			& = & \left | \int_{\tilde{t}}^{s}
			\left(\lambda_{1}^{\varepsilon}\left(x_{1}^{\varepsilon}\left(t\right),t\right) - \lambda_{2}^{\varepsilon}\left(x_{2}^{\varepsilon}\left(t\right),t\right)\right) dt \right | \nonumber \\
			& \le & \int_{\tilde{t}}^{s}
			\left |  \lambda_{1}^{\varepsilon}\left(x_{1}^{\varepsilon}\left(t\right),t\right) - \lambda_{2}^{\varepsilon}\left(x_{2}^{\varepsilon}\left(t\right),t\right) \right | dt \nonumber \\
			& \le & \int_{\tilde{t}}^{s}
			\left |  u^{\varepsilon}\left(x_{1}^{\varepsilon}\left(t\right),t\right) - u^{\varepsilon}\left(x_{2}^{\varepsilon}\left(t\right),t\right) \right | dt + C'\left(M\right) s \varepsilon^{2} \nonumber \\
			& \le & C''\left(M\right) \int_{\tilde{t}}^{s}
			\left |  x_{1}^{\varepsilon}\left(t\right) - x_{2}^{\varepsilon}\left(t\right) \right | dt + C'\left(M\right) s \varepsilon^{2}.
		\end{eqnarray}
		By $\rm{Gr\ddot{o}nwall}$'s inequality
		\begin{equation}
			\left | x_{1}^{\varepsilon}\left(s\right) - x_{2}^{\varepsilon}\left(s\right) \right | 
			\le C'\left(M\right) s \varepsilon^{2} \left(e^{C''\left(M\right) \left(s-\tilde{t}\right)}-1\right) 
			\le C\left(M\right) \varepsilon^{2} .
		\end{equation}
		Here, $C\left(M\right)$ is a uniform constant with respect to $\varepsilon$. Thus, we can see the different
		characteristic lines converge with the rate $O\left(\varepsilon^{2}\right)$ on $\mathbb{R}\times\left[0,t_{M}\right]$. Then we can have
		the following estimate of $z^{\varepsilon}$ and $u^{\varepsilon}$
		\begin{eqnarray}
			\left | z^{\varepsilon}\left(x_{2}^{\varepsilon}\left(t\right),t\right) - u^{\varepsilon}\left(x_{1}^{\varepsilon}\left(t\right),t\right) \right | 
			& \le & \left | u^{\varepsilon}\left(x_{2}^{\varepsilon}\left(t\right),t\right) - u^{\varepsilon}\left(x_{1}^{\varepsilon}\left(t\right),t\right) \right | + C'\left(M\right)\varepsilon^{2} \nonumber \\
			& \le & C''\left(M\right)\left | x_{1}^{\varepsilon}\left(s\right) - x_{2}^{\varepsilon}\left(s\right) \right | + C'\left(M\right)\varepsilon^{2} \nonumber \\
			& \le & \left(C''\left(M\right)C\left(M\right)+C'\left(M\right)\right)\varepsilon^{2} \nonumber \\
			& \le & C'''(M)\varepsilon^{2}.
		\end{eqnarray}
		It means that $z^{\varepsilon}$ converges to $u^{\varepsilon}$ with the rate of $\varepsilon^{2}$. By the same way, we can prove that $\lambda^{\varepsilon}_{1} \sim \lambda^{\varepsilon}_{2} \sim z^{\varepsilon} \sim u^{\varepsilon} \left(O\left(\varepsilon^{2}\right)\right)$ on $\mathbb{R}\times\left[0,t_{M}\right]$. Next, we consider the convergence rate of $u^{\varepsilon}$ to $\bar{u}$.
		\begin{eqnarray}
			\left | \bar{x}\left(s\right) - x_{1}^{\varepsilon}\left(s;\bar{x}\left(0\right),0\right) \right | & \le & \int_{0}^{s} \left | \bar{u}\left(\bar{x}\left(t\right),t\right)
			- \lambda_{1}^{\varepsilon} \left(x_{1}^{\varepsilon}\left(t;\bar{x}\left(0\right),0\right),t\right) \right | dt \nonumber \\
			& \le & \int_{0}^{s} \left | \bar{u}\left(\bar{x}\left(t\right),t\right)
			- u^{\varepsilon} \left(x_{1}^{\varepsilon}\left(t;\bar{x}\left(0\right),0\right),t\right) \right | dt + C'\left(M\right)s\varepsilon^{2} \nonumber \\
			& \le & \int_{0}^{s} \left | u_{0}\left(\bar{x}\left(0\right)\right)
			- u_{0} \left(\bar{x}\left(0\right)\right) \right | ds + C'\left(M\right)s\varepsilon^{2} \nonumber \\
			& \le & C\left(M\right)\varepsilon^{2}.
		\end{eqnarray}
		Then we can have
		the following estimate of $u^{\varepsilon}$ and $\bar{u}$
		\begin{eqnarray}\label{convergence of u and bar u}
			\left | u^{\varepsilon}\left(\tilde{x},\tilde{t}\right) - \bar{u}\left(\tilde{x},\tilde{t}\right) \right | 
			& \le & \left | u^{\varepsilon}\left(\tilde{x},\tilde{t}\right) - u^{\varepsilon}\left(x_{1}^{\varepsilon}\left(\tilde{t};\bar{x}\left(0\right),0\right),\tilde{t}\right) \right | 
			+ \left |  u^{\varepsilon}\left(x_{1}^{\varepsilon}\left(\tilde{t};\bar{x}\left(0\right),0\right),\tilde{t}\right) - \bar{u}\left(\bar{x}\left(\tilde{t}\right),\tilde{t}\right) \right | \nonumber \\
			& \le & \left | u^{\varepsilon}\left(\tilde{x},\tilde{t}\right) - u^{\varepsilon}\left(x_{1}^{\varepsilon}\left(\tilde{t};\bar{x}\left(0\right),0\right),\tilde{t}\right) \right | 
			+ \left |  u_{0}\left(\bar{x}\left(0\right) \right) - u_{0}\left(\bar{x}\left(0\right)\right) \right | \nonumber \\
			& \le & C''\left(M\right)\left | \tilde{x} - x_{1}^{\varepsilon}\left(\tilde{t};\bar{x}\left(0\right),0\right) \right | \nonumber \\
			& \le & C''\left(M\right)\left | \bar{x}\left(\tilde{t}\right) - x_{1}^{\varepsilon}\left(\tilde{t};\bar{x}\left(0\right),0\right) \right | \nonumber \\
			& \le & C''\left(M\right)C\left(M\right)\varepsilon^{2}.
		\end{eqnarray}
		We conclude that $ z^{\varepsilon} \sim u^{\varepsilon} \sim \bar{u} \left(O\left(\varepsilon^{2}\right)\right)$.
		
		Finally, we prove the convergence of discontinuous lines $\bar{x}(t)$ and $x^{\varepsilon}_{1}(t)$. Without loss of generality, we consider that there is only one discontinuity at $x=0$, and discuss the convergence of the discontinuity lines passing through 0. For the case of separable discontinuous points, we have similar results. By \eqref{convergence of u and bar u} and the uniform boundedness of $u_{x}^{\varepsilon}$, we have
		\begin{eqnarray}
			&& \left | \bar{x}\left(s;0,0\right) - x_{2}^{\varepsilon}\left(s;0,0\right) \right | \nonumber \\
			& \le & \int_{0}^{s} \left | \bar{u}\left(\bar{x}\left(t;0,0\right),t\right)
			- u^{\varepsilon} \left(x_{2}^{\varepsilon}\left(t;0,0\right),t\right)\right | dt \nonumber \\
			& \le & \int_{0}^{s} \left | \bar{u}\left(\bar{x}\left(t;0,0\right),t\right)
			- u^{\varepsilon} \left(\bar{x}\left(t;0,0\right),t\right)\right | 
			+ \left | u^{\varepsilon}\left(\bar{x}\left(t;0,0\right),t\right)
			- u^{\varepsilon} \left(x_{2}^{\varepsilon}\left(t;0,0\right),t\right)\right | dt \nonumber \\ 
			& \le & C(M)s\varepsilon^{2} + \int_{0}^{s} \left | u^{\varepsilon}\left(\bar{x}\left(t;0,0\right),t\right)
			- u^{\varepsilon} \left(x_{2}^{\varepsilon}\left(t;0,0\right),t\right)\right | dt \nonumber \\ 
			& \le & C(M)s\varepsilon^{2} + C'(M) \int_{0}^{s} \left | \bar{x}\left(t;0,0\right)
			- x_{2}^{\varepsilon}\left(t;0,0\right) \right | dt.
		\end{eqnarray}
		By $\rm{Gr\ddot{o}nwall}$'s inequality
		\begin{equation}
			\left | \bar{x}\left(s;0,0\right) - x_{2}^{\varepsilon}\left(s;0,0\right) \right | \le C(M)s\varepsilon^{2}e^{C'(M)s} \le C''(M)\varepsilon^{2}.
		\end{equation}
		We conclude that the discontinuous lines converges at the rate of $\varepsilon^{2}$.
		
		\section*{Acknowledgments }
		Tian-Yi Wang would like to thank Professor Pierangelo Marcati for introducing him to the Aw-Rascle Traffic Flow Model in the Russian seminar. He is deeply grateful for Professor Marcati's positive influence, invaluable help and support in all aspects, and fondly recalls the wonderful time working at GSSI and living in L'Aquila. The research of Tian-Ti Wang was supported in part by the National Science Foundation of China(Grants 12371223, 12061080).
		
		\section*{Conflict of interest}
		The authors  declare no conflict of interest.

	\end{proof}

	\vspace{0.5cm}
	
	\noindent \footnotesize{
		
		\noindent Zijie Deng: Wuhan University of Technology \\ 
		\emph{Email adress}: zijiedeng@whut.edu.cn \\

		\noindent Wenjian Peng: City University of Hong Kong \\
		\emph{Email adress}: wjpeng5-c@my.cityu.edu.hk \\

		\noindent Tian-Yi Wang: Wuhan University of Technology \\
		\emph{Email adress}: wangtianyi@amss.ac.cn \\

    	\noindent Haoran Zhang: Hong Kong Polytechnic University \\
    	\emph{Email adress}: mike-haoran.zhang@connect.polyu.hk \\

\end{document}